\documentclass[12pt,twoside,english,reqno,a4paper]{amsart}

\usepackage{listings,graphicx,amsmath,varioref,amscd,amssymb,color,bm,stmaryrd,amsthm,amsfonts,graphics,lineno}
\usepackage[english]{babel}

\usepackage{hyperref}

\newlength{\defbaselineskip}
\theoremstyle{plain}
\newtheorem{defin}{Definition}[section]
\newtheorem{theorem}[defin]{Theorem}
\newtheorem{prop}[defin]{Proposition}
\newtheorem{lemma}[defin]{Lemma}
\newtheorem{conj}[defin]{Conjecture}

\newtheorem{remark}[defin]{Remark}

\numberwithin{equation}{section}

\def\dis{\displaystyle}
\def\io{\int_{\Omega}} 
\def\supp{\text{\text{supp}}}

\DeclareMathOperator{\R}{\mathbb{R}}
\DeclareMathOperator{\N}{\mathbb{N}}

\newcommand{\car}[1]{\raise1pt\hbox{$\chi$}_{#1}}

\definecolor{sap}{RGB}{120,36,51}

\def\vp{\varphi}
\def\RN{\mathbb{R}^{N}}
\def\D{\nabla}

\def\linf{L^{\infty}(\Omega)}
\def\luno{L^{1}(\Omega)}
\def\lp'n{(L^{p'}(\Omega))^{N}}

\def\R{\mathbb{R}}
\def\N{\mathbb{N}}

\def\elle#1{L^{#1}(\Omega)}
\def\vp{\varphi}
\def\w{W_0^{1,2}(\Omega)}
\def\car#1{\chi_{_{{#1}}}}
\def\norma#1#2{\|#1\|_{\lower 4pt \hbox{$ \scriptstyle #2$ }}}

\def\un{u_{n}}
\def\um{u_{m}}
\def\vn{v_{n}}
\def\zn{z_{n}}

\newcommand{\LL}{\>\hbox{\vrule width.2pt \vbox to7pt{\vfill \hrule width7pt height.2pt}}\>}

\author[R. Durastanti]{Riccardo Durastanti 
\\Dipartimento di Scienze di Base e Applicate per l' Ingegneria, 
\\ ``Sapienza" Universit\`a di Roma, Via Scarpa 16, 00161 Roma, Italy
\\riccardo.durastanti@sbai.uniroma1.it}


\begin{document}

\title{Asymptotic behavior and existence of solutions for singular elliptic equations}

\maketitle

\begin{abstract}
We study the asymptotic behavior, as $\gamma$ tends to infinity, of solutions for the homogeneous Dirichlet problem associated to singular semilinear elliptic equations whose model is 
$$
-\Delta u=\frac{f(x)}{u^\gamma}\,\text{ in }\Omega,
$$
where $\Omega$ is an open, bounded subset of $\RN$ and $f$ is a bounded function. We deal with the existence of a limit equation under two different assumptions on $f$: either strictly positive on every compactly contained subset of $\Omega$ or only nonnegative. Through this study we deduce optimal existence results of positive solutions for the homogeneous Dirichlet problem associated to
$$
-\Delta v + \frac{|\nabla v|^2}{v} = f\,\text{ in }\Omega.
$$
\vskip 0.5\baselineskip
\end{abstract}

\tableofcontents

\section{Introduction}
\label{s1}

In recent years, existence, uniqueness and regularity of nonnegative solutions of the following semilinear singular problem have been widely studied:
\begin{equation}\label{pbintro}
\begin{cases}
\dis
-\Delta u = \frac{f}{u^{\gamma}} & \mbox{in $\Omega$,} \\
\hfill u = 0 \hfill & \mbox{on $\partial\Omega$.}
\end{cases}
\end{equation}
Here $\Omega$ is an open bounded subset of $\R^{N}$, with $N > 2$, $f$ is a nonnegative function belonging to some Lebesgue space and $\gamma > 0$. \\
Existence and uniqueness of a classical solution $u\in C^2(\Omega)\cap C(\overline{\Omega})$ of \eqref{pbintro} are proved in \cite{STU,CRT}, when $f$ is a positive H\"older continuous function in $\overline{\Omega}$ and $\Omega$ is a smooth domain. In the same framework, Lazer and McKenna in \cite{LAMC} prove that $u\in\w$ if and only if $\gamma<3$ and that, if $\gamma>1$, the solution does not belong to $C^1(\overline{\Omega})$, while in \cite{DELP}, under the weaker assumption that $f$ is only nonnegative and bounded, Del Pino proves existence and uniqueness of a positive distributional solution belonging to $C^1(\Omega)\cap C(\overline{\Omega})$. These results are generalized by Lair and Shaker in \cite{LASH}. \\
Existence of a positive distributional solution with data merely in $L^1(\Omega)$ is proved by Boccardo and Orsina in \cite{bo}. The authors show that this solution, if $\gamma<1$, belongs to an homogeneous Sobolev space larger than $\w$, if $\gamma=1$, it belongs to $\w$ and, finally, if $\gamma>1$, it belongs to $W^{1,2}_{loc}(\Omega)$ (see Theorem \ref{daBO} below). In the last case the boundary condition is assumed in a weaker sense, i.e. $\dis u^{\frac{\gamma+1}{2}}\in\w$. \\
Existence and regularity of solutions of \eqref{pbintro} with data in suitable Lebesgue space or with measure data are also studied in \cite{COCO,BGH1,BCT,ORPE,GMM2,CHDE}, while, in case of a nonlinear principal part, we refer to \cite{BGH2,DDO,OLIV}. We underline also the study of qualitative properties of solutions of \eqref{pbintro} contained in \cite{CMS,EFS}. \\
As concerns uniqueness of solutions of \eqref{pbintro} the literature is more limited. If a solution belongs to $\w$, uniqueness is proved in \cite{boca}, while in \cite{SZ} a necessary and sufficient condition in order to have $\w$ solutions is shown. Moreover we can find uniqueness results of solutions out of finite energy space in \cite{CASC,GMM1,OLPE}. \\

We observe that if we perform in \eqref{pbintro} the change of variable
$$
v = \frac{u^{\gamma+1}}{\gamma+1}\,,
$$
we formally transform \eqref{pbintro} into the quasilinear singular equation with singular and gradient quadric lower order term
\begin{equation}
\label{ipar}
\begin{cases}
\dis -\Delta v + \frac{\gamma}{\gamma+1}\,\frac{|\nabla v|^{2}}{v} = f & \mbox{in $\Omega$,} \\
\hfill v = 0 \hfill & \mbox{on $\partial\Omega$.}
\end{cases}
\end{equation}
Equation \eqref{ipar} is a particular case of the quasilinear singular equation
\begin{equation}\label{ipar2}
\begin{cases}
\dis -\Delta v + B\,\frac{|\nabla v|^{2}}{v^{\rho}} = f & \mbox{in $\Omega$,} \\
\hfill v = 0 \hfill & \mbox{on $\partial\Omega$,}
\end{cases}
\end{equation}
where $B$ and $\rho$ are positive real numbers. \\
One usually says that the quadratic growth in $\nabla v$ of \eqref{ipar2} is {\sl natural} as this growth is invariant under the simple change of variable $w=F(v)$, where $F$ is a smooth function. In this case the equation \eqref{ipar2} is also {\sl singular} since the lower order term is singular where the solution is zero.\\
 
Problem \eqref{ipar2} has been recently studied by several authors. Existence of classical solutions is studied by Porru and Vitolo in \cite{POVI}, while existence of a positive solution $v\in\w$ when $f$ is bounded and {\sl strictly positive} on every compactly contained subset of $\Omega$ and $0<\rho\leq 1$ is contained in \cite{AMA,ABMA}. Moreover, if $0<\rho<1$, Boccardo proves in \cite{B} existence of a positive weak solution under weaker assumptions on $f$, that is $f$ only {\sl nonnegative} and belonging to $\dis L^{\left(\frac{2^*}{\rho}\right)'}(\Omega)$. \\
As concerns the case $\rho\geq 1$, existence of positive weak solutions is proved in \cite{B,M-A} for $B < 1$ if $\rho = 1$ and $\dis f\in L^{\frac{2N}{N+2}}(\Omega)$ is {\sl nonnegative} in $\Omega$, while in \cite{ACLMOP} existence is proved for every $B > 0$ and for every $\rho < 2$ if the datum $f\in L^{\frac{2N}{N+2}}(\Omega)$ is {\sl strictly positive} on every compactly contained subset of $\Omega$. Moreover existence of positive solutions in the same framework of \cite{ACLMOP}, under a weaker assumption on $f$, that is $f$ {\sl strictly positive} on every compactly contained subset of a neighborhood of $\partial \Omega$, is proved in \cite{cmar}. Nonexistence results for positive solutions in $\w$ of \eqref{ipar2} are given, if $\rho >2$, in \cite{ACLMOP,ZWQ}. \\
The study of the uniqueness of weak solutions of \eqref{ipar2} is more limited in literature. We refer to \cite{ARSE} where uniqueness is proved if $0<\rho\leq 1$ and to \cite{CL} for $\rho\geq 1$. We underline also the multiplicity result of weak solutions contained in \cite{ZHOU}. \\
Without the aim to be exhaustive we also refer the reader to \cite{CR,GPS} in which the existence of solutions of \eqref{ipar2} is studied also in presence of sign-changing data, while we refer to \cite{BOP,DOP,XIYA} for the study of \eqref{ipar2} in the parabolic case. \\

Looking at the results for \eqref{ipar2}, the case $B = 1$ and $\rho = 1$ is a borderline case, requiring a stronger assumption on the datum in order to prove existence of positive weak solutions. In this paper we give an answer to the question whether this stronger assumption is really necessary, or if it is only technical. \\
From now onwards, we mean by $f$ {\sl strictly positive} a function $f$ strictly positive on every compactly contained subset of $\Omega$, that is for every subset $\omega$ compactly contained in $\Omega$ there exists a positive constant $c_\omega$ such that $f\geq c_\omega>0$ almost everywhere in $\omega$. \\
Since the case $B = 1$ and $\rho = 1$ can be seen as the limit case as $\gamma$ tends to infinity of equation \eqref{ipar}, and since this equation is connected to equation \eqref{pbintro}, one can try to study problem \eqref{ipar2}, in the borderline case $B = 1$ and $\rho = 1$, by looking at the asymptotic behavior, as $\gamma$ tends to infinity, of the solutions of \eqref{pbintro} under the assumption that $f$ is either {\sl nonnegative} or {\sl strictly positive}. \\

In this paper we prove, if $f$ is {\sl strictly positive} in $\Omega$, letting $\gamma$ tend to infinity, that there is no limit equation to \eqref{pbintro} and we find a positive solution to
\begin{equation}
\label{ipar3}
\begin{cases}
\dis -\Delta v + \frac{|\nabla v|^{2}}{v} = f & \mbox{in $\Omega$,} \\
\hfill v = 0 \hfill & \mbox{on $\partial\Omega$\,,}
\end{cases}
\end{equation}
recovering the existence result contained in \cite{AMA,ABMA,ACLMOP}. \\ 
If we assume $f$ only {\sl nonnegative}, more precisely zero in a neighborhood of $\partial \Omega$, we prove that there is a limit equation to \eqref{pbintro} and we give a one-dimensional example providing that the assumption $f$ {\sl strictly positive} cannot be relaxed in order to have a positive solution to \eqref{ipar3} as a limit of approximations. \\

Our results imply that the existence results contained in \cite{AMA,M-A,ABMA,ACLMOP,cmar} are sharp. \\

The plan of the paper is the following: in Section \ref{s2} we give the definitions of solution to our problems and we state the results that will be proved in the paper. In Section \ref{s3} we prove a priori estimates for the solutions of \eqref{pbintro} both from above and from below, that allow us to pass to the limit in \eqref{pbintro} and \eqref{ipar} as $\gamma$ tends to infinity. In Section \ref{s4} we pass to the limit in \eqref{pbintro} under the two different assumptions on $f$. In Section \ref{s5} we pass to the limit in \eqref{ipar}, in the case $f$ strictly positive, obtaining the existence of positive solutions of \eqref{ipar3}. In Section \ref{s6} we show, if $f$ is only nonnegative, the one-dimensional example of nonexistence of positive solutions to \eqref{ipar3} obtained by approximation. To conclude, in Section \ref{s7} we present some open problems. \\

\subsection*{Notations} 

Let $\Omega$ be an open and bounded subset of $\R^N$, with $N\geq 1$. We denote by $\partial \Omega$ its boundary, by $|A|$ the Lebesgue measure of a Lebesgue measurable subset $A$ of $\R^N$, and we define $\dis \mathrm{diam}(\Omega)=\sup\{|x-y| \,:\, x,y\in\Omega\}$.\\
By $C_c(\Omega)$ we mean the space of continuous functions with compact support in $\Omega$ and by $C_0(\Omega)$ the space of continuous functions in $\Omega$ that are zero on $\partial\Omega$. Analogously, if $k\geq 1$, $C^k_c(\Omega)$ (resp. $C^k_0(\Omega)$) is the space of $C^k$ functions with compact support in $\Omega$ (resp. $C^k$ functions that are zero on $\partial\Omega$). \\
If no otherwise specified, we will denote by $C$ several constants whose value may change from line to line. These values will only depend on the data (for instance $C$ may depend on $\Omega$, $N$) but they will never depend on the indexes of the sequences we will introduce. \\
Moreover, for any $q>1$, $q'$ will be the H\"older conjugate exponent of $q$, while for any $1\leq p<N$, $p^*=\frac{Np}{N-p}$ will be the Sobolev conjugate exponent of $p$. We will also denote by $\epsilon(n)$ any quantity such that
$$\limsup_{n\rightarrow\infty}\epsilon(n)=0\,.$$
We will use the following well-known functions defined for a fixed $k>0$
$$T_k(s)=\max (-k,\min (s,k)) \quad \text{and}\quad G_k(s)=(|s|-k)^+ \operatorname{sign}(s),$$
with $s\in\R$. \\
We also mention the definition of the Gamma function
\begin{equation}
\label{gammadef}
\Gamma(z)\,=\,\int_0^{+\infty} t^{z-1} \mathrm{e}^{-t}\,dt\,,
\end{equation}
where $z$ is a complex number with positive real part, recalling that $\Gamma(1)=1$ and $\dis \Gamma\left(\frac{1}{2}\right)=\sqrt{\pi}$. \\
Finally we define $\phi_\lambda:\R\to\R$, with $\lambda >0$, the following function
\begin{equation}
\label{defphilam}
\phi_\lambda(s)\,=\,s\,\mathrm{e}^{\lambda s^2}\,.
\end{equation}
In what follows we will use that for every $a,b>0$ we have, if $\dis \lambda > \frac{b^2}{4a^2}$, that
\begin{equation}
\label{defphilam2}
a\,\phi'_\lambda(s)-b\,|\phi_\lambda(s)|\,\geq\,\frac{a}{2}\,.
\end{equation}
 
\section{Main assumptions and statement of the results}
\label{s2}

Let $M(x)$ be a matrix which satisfies, for some positive constants $0<\alpha\leq \beta$, for almost every $x\in\Omega$ and for every $\xi \in \R^{N}$ the following assumptions:
\begin{equation}
\label{albe}
M(x)\,\xi\cdot\xi \geq \alpha |\xi|^{2} \qquad \text{ and } \qquad |M(x)| \leq \beta\,.
\end{equation}
Let $\gamma>0$ be a real number. We consider the following semilinear elliptic problem with a singular nonlinearity 
\begin{equation}
\label{pbo}
\begin{cases}
\dis -{\rm div}(M(x)\nabla u) = \frac{f}{u^{\gamma}} & \mbox{in $\Omega$,} \\
\hfil u>0  & \mbox{in $\Omega$,} \\
\hfil u = 0  & \mbox{on $\partial\Omega$.}
\end{cases}
\end{equation}
To deal with existence for solutions to problem \eqref{pbo} we give the following definition of distributional solution contained in \cite{bo}.
\begin{defin}
A function $u$ in $W^{1,1}_{loc}(\Omega)$ such that 
\begin{equation*}
\begin{cases}
u\in W^{1,1}_0(\Omega) & \text{if $\gamma < 1$},\\
u^{\frac{\gamma+1}{2}} \in W^{1,1}_0(\Omega) & \text{if $\gamma \ge 1$},
\end{cases}
\end{equation*}
is a distributional solution of \eqref{pbo} if the following conditions are satisfied:
$$
\forall \omega\subset\subset \Omega\,\,\,\exists\,c_{\omega,\gamma}\,:\,u\,\geq \,c_{\omega,\gamma}\,>\,0\,\text{ in }\omega\,,
$$
and
$$
\io M(x)\nabla u\cdot\nabla\varphi\,=\,\io \frac{f\,\varphi}{u^\gamma}\,,\quad\forall\,\varphi\in C^1_c(\Omega)\,.
$$
\end{defin}
We underline that, if $\gamma>1$, the condition $\dis u^{\frac{\gamma+1}{2}} \in W^{1,1}_0(\Omega)$ gives meaning to the boundary condition of \eqref{pbo}. \\
We start studying  the asymptotic behavior of the sequence $\{\un\}$ of solutions to problem \eqref{pbo}, with $\gamma=n$. Our results are the following:
\begin{theorem}\label{main1}\sl
Let $f$ be a nonnegative $\elle\infty$ function. Suppose that there exists $\omega \subset\subset \Omega$ such that $f = 0$ in $\Omega \setminus \omega$, and such that for every $\omega' \subset \subset \omega$ there exists $c_{\omega'} > 0$ such that $f \geq c_{\omega'}$ in $\omega'$. Let $\{\un\}$ be a sequence of distributional solutions of
\begin{equation}\label{pbn}
\begin{cases}
\dis
-{\rm div}(M(x)\nabla \un) = \frac{f(x)}{\un^{n}} & \mbox{in $\Omega$,} \\
\hfill \un = 0 \hfill & \mbox{on $\partial\Omega$.}
\end{cases}
\end{equation}
Then $\{\un\}$ is bounded in $\elle\infty$, so that it converges, up to subsequences, to a bounded function $u$ which is identically equal to 1 almost everywhere in $\omega$. Furthermore, the sequence $\{f(x)/\un^{n}\}$ is bounded in $\elle1$, and if $\mu$ is the $*$-weak limit in the sense of measures of $f(x)/\un^{n}$, $\mu$ is concentrated on $\partial\omega$, and $u$ in $W^{1,2}_0(\Omega)$ is the solution of
\begin{equation}\label{pblim0}
\begin{cases}
-{\rm div}(M(x)\nabla u) = \mu & \mbox{in $\Omega$,} \\
\hfill u = 0 \hfill & \mbox{on $\partial\Omega$.}
\end{cases}
\end{equation}
\end{theorem}

\begin{theorem}\label{main2}\sl
Let $f$ be a function belonging to $L^{\infty}(\Omega)$ such that for every $\omega \subset\subset \Omega$ there exists $c_{\omega} > 0$ such that $f \geq c_{\omega}$ in $\omega$. Let $\{\omega_{n}\}$ be an increasing sequence of compactly contained subsets of $\Omega$ such that their union is $\Omega$, and let $\un$ be the distributional solution of
\begin{equation}\label{pbn2}
\begin{cases}
\dis
-{\rm div}(M(x)\nabla \un) = \frac{f(x)\,\chi_{\omega_{n}}}{\un^{n}} & \mbox{in $\Omega$,} \\
\hfill \un = 0 \hfill & \mbox{on $\partial\Omega$.}
\end{cases}
\end{equation}
Then $\{\un\}$ is bounded in $\elle\infty$, so that it converges, up to subsequences, to a bounded function $u$, which is identically equal to~1 almost everywhere in $\Omega$. Moreover, the sequence $\{f(x)\chi_{\omega_{n}}/\un^{n}\}$ is unbounded in $\elle1$, and there is no limit equation for $u$.
\end{theorem}

If $M(x)\equiv I$, we have that $\dis \left\{v_n=\frac{u_n^{n+1}}{n+1}\right\}$ is a sequence of distributional solutions to the following problem
\begin{equation}\label{parn}
\begin{cases}
\dis -\Delta v_n + \frac{n}{n+1}\,\frac{|\nabla v_n|^{2}}{v_n} = f(x) & \mbox{in $\Omega$,} \\
\hfil v_n = 0 & \mbox{on $\partial\Omega$.}
\end{cases}
\end{equation}
To be complete we give the definitions of distributional and weak solution for quasilinear elliptic equations with singular and gradient quadratic lower order term whose model is 
\begin{equation}
\label{24pb}
\begin{cases}
\dis -\Delta v + B\,\frac{|\nabla v|^2}{v} = f & \mbox{in $\Omega$,} \\
\hfil v = 0  & \mbox{on $\partial\Omega$,}
\end{cases}
\end{equation}
where $B>0$.
\begin{defin}
A function $v$ in $W^{1,2}_0(\Omega)$ is a weak solution of \eqref{24pb} if the following conditions are satisfied:
\begin{itemize}
\item[i)] $v>0$ almost everywhere in $\Omega$,
\item[ii)] $\dis \frac{|\nabla v|^2}{v}$ belongs to $L^1(\Omega)$,
\item[iii)] it holds
$$
\io \nabla v\cdot\nabla\varphi + B\,\io \frac{|\nabla v|^{2}}{v}\,\varphi =\io f\,\varphi\,, \quad \forall\, \varphi\in \w\cap\linf\,.
$$
\end{itemize}
\end{defin}
\begin{defin}
A function $v$ in $W^{1,1}_0(\Omega)$ is a distributional solution of \eqref{24pb} if the following conditions are satisfied:
\begin{itemize}
\item[i)] $v>0$ almost everywhere in $\Omega$,
\item[ii)] $\dis \frac{|\nabla v|^2}{v}$ belongs to $L^1(\Omega)$,
\item[iii)] it holds
$$
\io \nabla v\cdot\nabla\varphi + B\,\io \frac{|\nabla v|^{2}}{v}\,\varphi =\io f\,\varphi\,, \quad \forall\, \varphi\in C^1_c(\Omega)\,.
$$
\end{itemize}
\end{defin}
By passing to the limit in \eqref{parn}, we prove, in the case $f$ {\sl strictly positive}, the following existence theorem of weak solution to problem \eqref{pbv4}.
\begin{theorem}
\label{main3}
Let $f$ be a $\elle\infty$ function such that for every $\omega \subset\subset \Omega$ there exists $c_{\omega} > 0$ such that $f \geq c_{\omega}$ in $\omega$. Then $\displaystyle\left\{v_n=\frac{u_n^n+1}{n+1}\right\}$ is bounded in $\w\cap \linf$, so that it converges, up to subsequences, to a bounded nonnegative function $v$ which is a weak solution of 
\begin{equation}\label{pbv4}
\begin{cases}
\dis -\Delta v +\frac{|\nabla v|^{2}}{v} = f & \mbox{in $\Omega$,} \\
\hfill v = 0 \hfill & \mbox{on $\partial\Omega$.}
\end{cases}
\end{equation}
\end{theorem}

On the other hand, if $f$ is zero in a neighborhood of $\partial \Omega$, we show by a one-dimensional explicit example that the function obtained as limit of our approximation is zero in a subset of $\Omega$ with strictly positive measure. In other words, we will prove the following result.

\begin{theorem}
\label{main4}
Let $\Omega=(-2,2)$ and $\omega=(-1,1)$. Let $u_n$ in $W^{1,2}_0((-2,2))$ be the weak solution of
\begin{equation}
\label{oned00}
\begin{cases}
\dis -u_n''(t) =\frac{\chi_{(-1,1)}}{u_n^n} & \mbox{ in } (-2,2)\,, \\
u_n(\pm 2)\, =\,0\,.
\end{cases}
\end{equation}
Let $\dis v_n=\frac{u_n^{n+1}}{n+1}$ be a weak solution of
\begin{equation}\label{pbv5}
\begin{cases}
\dis -v_n'' +\frac{n}{n+1}\frac{|v_n'|^{2}}{v_n}\, =\,\chi_{(-1,1)}  & \mbox{in $(-2,2)$,} \\
v_n(\pm 2)\, =\, 0\,,
\end{cases}
\end{equation}
then $\{v_n\}$ weakly converges to a function $v$ in $W^{1,2}_0((-2,2))$ and $v$, belonging to $C^\infty_0((-1,1))$, is a classical solution of
\begin{equation}\label{pbv6}
\begin{cases}
\dis -v'' +\frac{|v'|^{2}}{v} \, =\, 1 & \mbox{in $(-1,1)$,} \\
v(\pm 1)\, = \,0 \,.
\end{cases}
\end{equation}
Moreover $\dis v(t)=\frac{2}{\pi^2}\cos^2\left(\frac{\pi}{2}t\right)$ in $(-1,1)$ and $v(t)\equiv 0$ in $[-2,-1] \cup [1,2]$.
\end{theorem}
 
\begin{remark}
\label{sharp}
It follows from Theorem \ref{main4} that if $f$ is only {\sl nonnegative} we cannot obtain by approximation a positive solution of \eqref{pbv4}. This implies that the assumption $f$ {\sl strictly positive} is necessary (and not only technical) to have positive solutions on the whole $\Omega$ to problem \eqref{pbv4}. Hence the existence results contained in \cite{AMA,M-A,ABMA,ACLMOP,cmar} are optimal.
\end{remark}

\section{Estimates from above and from below}
\label{s3}

In \cite{bo}, existence results for distributional solutions of \eqref{pbo} have been proved. To be more precise, we have the following theorem in the case $\gamma>1$.
\begin{theorem}\label{daBO}\sl
Let $\gamma > 1$, and let $f$ be in $\elle\infty$, with $f \geq 0$ in $\Omega$, $f$ not identically zero. Then there exists a distributional solution $u$ of \eqref{pbo}, with $u$ in $W^{1,2}_{{\rm loc}}(\Omega) \cap \elle\infty$. Moreover we can extend the class of test functions in the sense that
\begin{equation}\label{defsol}
\io M(x)\D u \cdot \D\vp = \io \frac{f\,\vp}{u^{\gamma}},
\qquad
\forall \vp \in W^{1,2}_{0}(\Omega) \text{ with compact support}.
\end{equation}
\end{theorem}

\begin{proof}[{\sl Sketch of the proof of Theorem \ref{daBO}}]

Following \cite{bo}, let $m$ in $\N$ and consider the approximated problems
\begin{equation}\label{pnm}
\begin{cases}
\dis
-{\rm div}(M(x)\D \um) = \frac{f}{(\um + \frac1m)^{\gamma}} & \mbox{in $\Omega$,} \\
\hfill \um > 0 \hfill & \mbox{in $\Omega$,} \\
\hfill \um = 0 \hfill & \mbox{on $\partial\Omega$.}
\end{cases}
\end{equation}
The existence of a solution $\um$ can be easily proved by means of the Schauder fixed point theorem. Since the sequence $\dis g_{m}(s) = \frac{1}{(s + \frac1m)^{\gamma}}$ is increasing in $m$, standard elliptic estimates imply that the sequence $\{\um\}$ is increasing, so that $\um \geq u_{1}$, and there exists the pointwise limit $u$ of $\um$. Since (by the maximum principle) for every $\omega \subset\subset \Omega$ there exists $c_{\omega,\gamma} > 0$ such that $u_{1} \geq c_{\omega,\gamma}$ in $\omega$, it then follows that $\um$ (and so $u$) has the same property. 

Choosing $\um^{\gamma}$ as test function in \eqref{pnm} we obtain, using \eqref{albe}, that
$$
\frac{4\alpha\gamma}{(\gamma+1)^{2}} \int_{\Omega} |\nabla \um^{\frac{\gamma+1}{2}}|^{2}
\leq
\gamma \int_{\Omega} M(x) \nabla\um \cdot \nabla\um\,\um^{\gamma-1}
=
\int_{\Omega}\,\frac{f \,\um^{\gamma}}{(\um + \frac1m)^{\gamma}}
\leq
\int_{\Omega}\,f\,.
$$
Therefore, $\{\um^{\frac{\gamma+1}{2}}\}$ is bounded in $\w$. Choosing $\um\,\vp^{2}$ as test function in \eqref{pnm}, with $\vp$ in $C^{1}_{0}(\Omega)$, we obtain, using again \eqref{albe},
$$
\alpha\,\int_{\Omega}|\nabla\um|^{2}\,\vp^{2}
+
2 \int_{\Omega}\,M(x) \nabla\um \nabla \vp \, \um\,\vp
\leq
\int_{\Omega}\,\frac{f\,\um\,\vp^{2}}{(\um + \frac1m)^{\gamma}}\,.
$$
Hence, if $\omega = \{\vp \neq 0\}$, recalling that $\um \geq c_{\omega,\gamma} > 0$ in $\omega$, we have, by Young's inequality,
$$
\alpha\,\int_{\Omega} |\nabla\um|^{2}\,\vp^{2}
\leq
\frac{\alpha}{2}\,\int_{\Omega} |\nabla\um|^{2}\,\vp^{2}
+
C\,\int_{\Omega}\,|\nabla\vp|^{2}\,\um^{2}
+
\frac{\norma{f\vp^2}{\elle\infty}}{c_{\omega,\gamma}^{\gamma}}\int_{\Omega}\,\um\,.
$$
Since $\um$ is bounded in $\elle2$ (recall that $\um^{\frac{\gamma+1}{2}}$ is bounded in $\w$, so that $\um^{\gamma+1}$ is bounded in $\elle1$ by Poincar\'e inequality, and that $\gamma > 1$), we thus have
$$
\int_{\Omega}\,|\nabla\um|^{2}\,\vp^{2} \leq C\,,
$$
so that the sequence $\{\um\}$ is bounded in $W^{1,2}_{{\rm loc}}(\Omega)$. Let now $k > 1$ and choose $G_{k}(\um)$ as test function in \eqref{pnm}. We obtain, using \eqref{albe},
$$
\alpha \int_{\Omega} |\nabla G_{k}(\um)|^{2} \leq \int_{\Omega}\,\frac{f\,G_{k}(\um)}{(\um + \frac1m)^{\gamma}}
\leq
\frac{1}{k^{\gamma}}\,\int_{\Omega}\,f\,G_{k}(\um)\,,
$$
so that
$$
\alpha \int_{\Omega} |\nabla G_{k}(\um)|^{2} \leq \int_{\Omega}\,f\,G_{k}(\um)\,, \qquad \forall k \geq 1\,.
$$
Starting from this inequality, and reasoning as in Theorem 4.2 of \cite{sta}, we can prove that $\um$ is uniformly bounded in $\elle\infty$, so that $u$ belongs to $\elle\infty$ as well.

Once we have the {\sl a priori} estimates on $\um$, we can pass to the limit in the approximate equation with test functions $\vp$ in $W^{1,2}_{0}(\Omega)$ with compact support; indeed
$$
\lim_{m \to +\infty}\,\int_{\Omega} M(x)\,\nabla\um \cdot \nabla \vp = \int_{\Omega} M(x)\,\nabla u \cdot \nabla\vp\,,
$$
since $\um$ is weakly convergent to $u$ in $W^{1,2}_{{\rm loc}}(\Omega)$, and
$$
\lim_{m \to +\infty}\,\int_{\Omega}\,\frac{f\,\vp}{(\um + \frac1m)^{\gamma}} = \int_{\Omega}\,\frac{f\,\vp}{u^{\gamma}}\,,
$$
by the Lebesgue theorem, since $\um \geq c_{\{\vp \neq 0\},\gamma} > 0$ on the support of $\vp$.
\end{proof}
Since the formulation of distributional solution for \eqref{pbo} is not suitable for our purposes, we are going to better specify the class of test functions which are admissible for the problem \eqref{pbo} to obtain estimates from above for $u$.
We start with the following theorem.
\begin{theorem}\label{otherprop}\sl
The solution $u$ of \eqref{pbo} given by Theorem \ref{daBO} is such that:
\begin{itemize}
	\item[i)] $u^{\gamma+1}$ belongs to $\w$;
	
\item[ii)] 	
\begin{equation}\label{newform}
\int_{\Omega} M(x)\,\nabla \Big( \frac{u^{\gamma+1}}{\gamma+1} \Big) \cdot \nabla v
\leq
\int_{\Omega} f\,v\,,
\qquad
\forall v \in \w\,,\ v \geq 0\,;
\end{equation}

	\item[iii)] 
\begin{equation}\label{stimalinfty}
\norma{u}{\elle\infty} \leq [C\,(\gamma+1)\,\norma{f}{\elle\infty}]^{\frac{1}{\gamma+1}}\,,
\end{equation}
for some constant $C > 0$, independent on $\gamma$.
\end{itemize}
\end{theorem}

\begin{proof}[{\sl Proof.}]
We begin by observing that, using the boundedness in $\elle\infty$ of the sequence $\um$ of solutions of \eqref{pnm}, and the boundedness of $\um^{\frac{\gamma+1}{2}}$ in $\w$, the sequence $\um^{p}$ is bounded in $\w$ for every $p \geq \frac{\gamma+1}{2}$. In particular, $\{\um^{\gamma+1}\}$ is bounded in $\w$. This yields that $u^{\gamma+1}$ belongs to $\w$ as well; i.e., i) is proved.

We now fix a positive $\vp$ in $C^{1}_{0}(\Omega)$ and take $\um^{\gamma}\,\vp$ as test function in \eqref{pnm}. We obtain
$$
\gamma\io M(x)\D\um \cdot \D\um \, \um^{\gamma-1}\vp
+
\io M(x) \D\um \cdot \D\vp \,\um^{\gamma}
\leq
\io f\,\vp.
$$
Dropping the first term (which is positive), we obtain
$$
\io M(x) \D \Big(\frac{\um^{\gamma+1}}{\gamma+1}\Big) \cdot \D\vp
\leq
\io f\,\vp\,.
$$
Letting $m$ tend to infinity, and using the boundedness of $\um^{\gamma+1}$ in $\w$, we obtain
$$
\io M(x) \D \Big(\frac{u^{\gamma+1}}{\gamma+1}\Big) \cdot \D\vp
\leq
\io f\,\vp\,,
\qquad
\forall \vp \in C^{1}_{0}(\Omega)\,,\ \vp \geq 0\,.
$$
Since $u^{\gamma+1}$ belongs to $\w$, we obtain by density
$$
\io M(x) \D \Big(\frac{u^{\gamma+1}}{\gamma+1}\Big) \cdot \D v
\leq
\io f\,v\,,
\qquad
\forall v \in \w\,,\ v \geq 0\,,
$$
which is \eqref{newform}.
We now choose 
$$
v = G_{k}\Big(\frac{u^{\gamma+1}}{\gamma+1}\Big)\,,
$$
as test function in \eqref{newform} (recall that $u \geq 0$, so that $v \geq 0$ as well). We obtain, setting $A_{\gamma}(k) = \{ u^{\gamma+1} \geq (\gamma+1)\,k\} = \{v \geq 0\}$,
$$
\int_{A_{\gamma}(k)}\,M(x)\D \Big(\frac{u^{\gamma+1}}{\gamma+1}\Big) \cdot \D G_{k}\Big(\frac{u^{\gamma+1}}{\gamma+1}\Big)
\leq
\int_{A_{\gamma}(k)}\,f\,G_{k}\Big(\frac{u^{\gamma+1}}{\gamma+1}\Big)\,.
$$
Recalling \eqref{albe} we therefore have
$$
\alpha\int_{A_{\gamma}(k)}\,\Big| \D G_{k}\Big(\frac{u^{\gamma+1}}{\gamma+1}\Big)\Big|^{2}
\leq
\int_{A_{\gamma}(k)}\,f\,G_{k}\Big(\frac{u^{\gamma+1}}{\gamma+1}\Big)\,.
$$
From this inequality, reasoning once again as in \cite{sta}, we obtain that there exists $C > 0$ such that
$$
\Big\|\frac{u^{\gamma+1}}{\gamma+1}\Big\|_{\elle\infty}
\leq
C\,\norma{f}{\elle\infty}\,,
$$
which then yields \eqref{stimalinfty}.
\end{proof}

\begin{remark}\label{ben}\sl
We observe that if we also assume that $\omega = \{f > 0\}$ is compactly contained in $\Omega$ in Theorem \ref{daBO}, then $u$ belongs to $W^{1,2}_0(\Omega)$ and $\dis \frac{f}{u^\gamma}$ belongs to $L^1(\Omega)$. As a matter of fact, taking $u_m$ as test function in \eqref{pnm}, we have
$$
\alpha\io |\D u_m|^2 \leq \int_{\Omega}\,\frac{f \,\um}{(\um + \frac1m)^{\gamma}} \leq \frac{\norma{f}{\elle\infty}}{c_{\omega,\gamma}^{\gamma-1}}\,,
$$
so that $u$ belongs to $\w$. Moreover, using the Lebesgue theorem and that $u_m \geq c_{\omega,\gamma}$, we deduce that $\dis \frac{f}{u_m^\gamma}$ strongly converges to $\dis \frac{f}{u^\gamma}$ in $L^1(\Omega)$. As a consequence we can extend the class of test functions for \eqref{defsol} to $\w$.
\end{remark}

\begin{remark}\label{bocas}\sl
Under the assumptions of Remark \ref{ben}, thanks to the results contained in \cite{boca}, it follows that $u$ is the unique weak solution of \eqref{pbo}.
\end{remark}

From now on, $\gamma = n$, and we will denote by $\un$ the solution of \eqref{pbn}; therefore, by the results of Theorem \ref{otherprop}, we have that $\un^{n+1}$ belongs to $\w \cap \elle\infty$, and that
$$
\norma{\un}{\elle\infty} \leq (C(n+1)\norma{f}{\elle\infty})^{\frac{1}{n+1}}\,,
$$
which in particular implies that
\begin{equation}\label{fromabove}
\limsup_{n \to +\infty}\,\norma{\un}{\elle\infty} \leq 1\,.
\end{equation}

We now consider the estimates from below on the sequence $\{\un\}$. We first need to enunciate two technical lemmas that we will use during the proof of these estimates.
\begin{lemma}\label{stamp}\sl
Let $m(j,r):[0,+\infty)\times [0,R_0)\to [0,+\infty)$ be a function such that $m(\cdot,r)$ is nonincreasing and $m(j,\cdot)$ is nondecreasing. Moreover, suppose that there exist $k_0\geq 0$, $C, \nu, \delta > 0$ and $\mu>1$ satisfying
$$
m(j,r)\leq C\frac{m(k,R)^\mu}{(j-k)^\nu (R-r)^\delta} \qquad \forall \,\, j>k \geq k_0, \,\,0 \leq r<R<R_0. 
$$
Then, for every $0<\sigma<1$, there exists $d>0$ such that
$$
m(k_0+d,(1-\sigma)R_0)=0,
$$
where $\dis d^\nu=\frac{2^{(\nu+\delta)\frac{\mu}{\mu-1}}Cm(k_0,R_0)^{\mu-1}}{\sigma^\delta R_0^\delta}$.
\end{lemma}

\begin{proof}[{\sl Proof.}]
See \cite{STA2}.
\end{proof}

\begin{lemma}\label{leo}\sl
Let $g:[0,+\infty)\to[0,+\infty)$ be a continuous and increasing function, with $g(0)=0$, such that 
$$
t\in(0,+\infty)\mapsto\frac{g(t)}{t} \text{ is increasing and } \int^{+\infty}\frac{1}{\sqrt{tg(t)}} < +\infty.
$$
Then, for any $C>0$ and $\delta\geq 0$, there exists a function $\vp:[0,1]\to[0,1]$ depending on $g,C,\delta$ with $\vp\in C^1([0,1])$, $\sqrt{\vp}\in C^1([0,1])$, $\vp(0)=\vp'(0)=0$, $\vp(1)=1$, $\vp(\sigma)>0$ for every $\sigma>0$ and satisfying 
$$
t^{\delta+1}\frac{\vp'(\sigma)^2}{\vp(\sigma)}\leq \frac{1}{C}t^\delta g(t)\vp(\sigma)+1, \qquad \forall \,\, 0\leq\sigma \leq 1,\,\, t\geq 0.
$$
\end{lemma}

\begin{proof}[{\sl Proof.}]
See \cite{L}, Lemma 1.1.
\end{proof}

We are ready to prove the estimates from below.

\begin{theorem}\label{dasotto}\sl
Let $\un$ be the solution of \eqref{pbn} given by Theorem \ref{daBO}, and let $\omega \subset\subset \Omega$ be such that for every $\omega' \subset \subset \omega$ there exists $c_{\omega'} > 0$ satisfying $f \geq c_{\omega'}$ in $\omega'$. Then there exists $M_{\omega'} > 0$ such that
\begin{equation}\label{frombelow}
\un \geq (n+1)^{\frac{1}{n+1}}{\rm e}^{-\frac{M_{\omega'}}{n+1}}
\qquad
\mbox{in $\omega'$.}
\end{equation}
\end{theorem}

\begin{proof}[{\sl Proof.}]
Let $\omega''\subset\subset\omega' \subset\subset \omega$, by the assumptions we have that
\begin{equation}\label{mo}
m_{\omega'} = \inf_{x \in \omega'}\,f(x) > 0\,.
\end{equation}
Let $\eta$ in $C^{1}_{0}(\Omega)$ be such that
$$
\eta(x) = 
\begin{cases}
1 & \mbox{in $\omega''$,} \\
0 & \mbox{in $\Omega \setminus \overline{\omega'}$.}
\end{cases}
$$
We consider the function $\vp\in C^1([0,1])$ given by Lemma \ref{leo}, in correspondence of $g(t)=\rm{e}^t-1$, $\delta=1$ and of an arbitrary constant $C>0$. Define 
$$
\xi(x)=\sqrt{\vp(\eta(x))}\in C^{1}_{0}(\Omega)\,,
$$
$$
\zn = -\log \Big(\frac{\un^{n+1}}{n+1} \Big)\,,
$$
and, for $k > 0$,
$$
\vn = \frac{G_{k}(\zn^{+})}{\un}\,.
$$
Note that $\vn \geq 0$ is well defined, since where $\zn^{+} > k$ one has $\un \neq 0$. We have
\begin{equation}\label{gradxi}
\D\xi=\frac{\vp'(\eta)}{2\sqrt{\vp(\eta)}}\,\D\eta\,.
\end{equation}
Since 
$$
\D\zn = -\frac{(n+1)\D\un}{\un}
$$
we obtain
$$
\D \vn
=
-
\frac{\D\un}{\un^{2}}\,\,G_{k}(\zn^{+})
+
\frac{1}{\un}\,\D\zn\,\car{A_{n}(k)}
=
-
\frac{\D\un}{\un^{2}}\,\,G_{k}(\zn^{+})
-
\frac{(n+1)\D\un}{\un^{2}}\car{A_{n}(k)}\,,
$$
where $A_{n}(k) = \{ \zn^{+} \geq k\} = \{G_{k}(\zn^{+}) \neq 0\}$. Therefore, since $\un$ belongs to $W^{1,2}_{{\rm loc}}(\Omega) \cap \elle\infty$ and it is locally positive, $\zn$ and $\vn$ belong to $W^{1,2}_{{\rm loc}}(\Omega)$. Consequently the positive function $\vn\,\xi^{2}$ belongs to $\w$, has compact support and can be chosen as test function in \eqref{defsol}, with $\gamma=n$, to obtain
$$
\begin{array}{l}
\dis
-
\int_{A_{n}(k)} M(x)\D\un \cdot \D\un\,\frac{G_{k}(\zn^{+})\,\xi^{2}}{\un^{2}}
-
\int_{A_{n}(k)} M(x)\D\un \cdot \D\un\,\frac{(n+1)\,\xi^{2}}{\un^{2}}
\\
\dis
\qquad
+
2\int_{A_{n}(k)} M(x)\D\un \cdot \D\xi \,\frac{G_{k}(\zn^{+})\,\xi}{\un}
=
\int_{A_{n}(k)}\,\frac{f\,G_{k}(\zn^{+})\,\xi^{2}}{\un^{n+1}}\,.
\end{array}
$$
Since
$$
\frac{n+1}{\un^{n+1}} = {\rm e}^{\zn}\,,
$$
the previous identity can be rewritten as
$$
\begin{array}{l}
\dis
-
\frac{1}{n+1}\int_{A_{n}(k)} M(x)\D\zn \cdot \D\zn \,G_{k}(\zn^{+})\,\xi^{2}
-
\int_{A_{n}(k)}M(x)\D\zn \cdot \D\zn\,\xi^{2}
\\
\dis
\qquad
-
2 \int_{A_{n}(k)} M(x)\D\zn \cdot \D\xi \,G_{k}(\zn^{+})\,\xi
=
\int_{A_{n}(k)}\,f\,{\rm e}^{\zn^{+}}G_{k}(\zn^{+})\,\xi^{2}\,.
\end{array}
$$
Since the first term is negative, we have, using \eqref{albe} and \eqref{mo}, as well as the fact that $G_{k}(s^{+}) \leq s^{+}$, that
$$
\alpha\int_{A_{n}(k)}|\D\zn|^{2}\,\xi^{2}
+
m_{\omega'}\int_{A_{n}(k)} {\rm e}^{G_{k}(\zn^{+})} G_{k}(\zn^{+})\,\xi^{2}
\leq
2\beta \int_{A_{n}(k)} |\D\zn||\D\xi| G_{k}(\zn^{+})\, \xi\,.
$$
Using Young's inequality in the right hand side, we have
$$
2\beta \int_{A_{n}(k)} |\D\zn||\D\xi| G_{k}(\zn^{+})\,\xi
\leq
\frac{\alpha}{2}\int_{A_{n}(k)}|\D\zn|^{2}\,\xi^{2}
+
\frac{2\beta^2}{\alpha}\,\int_{A_{n}(k)} |\D\xi|^{2}\,G_{k}(\zn^{+})^{2}\,,
$$
so that we have
$$
\frac{\alpha}{2}\int_{A_{n}(k)}|\D G_k(\zn^{+})|^{2}\,\xi^{2}
+
m_{\omega'}\int_{A_{n}(k)} {\rm e}^{G_{k}(\zn^{+})} G_{k}(\zn^{+})\,\xi^{2}
\leq
\frac{2\beta^2}{\alpha}\,\int_{A_{n}(k)} |\D\xi|^{2}\,G_{k}(\zn^{+})^{2}\,.
$$
Observing that
$$
\frac{\alpha}{4}|\D(G_k(\zn^{+})\xi)|^2\leq \frac{\alpha}{2}|\D G_k(\zn^{+})|^{2}\,\xi^{2}+\frac{\alpha}{2}|\D\xi|^2\,G_k(\zn^{+})^2\,,
$$
we obtain
$$
\frac{\alpha}{4}\int_{A_{n}(k)}|\D(G_k(\zn^{+})\xi)|^{2}
+
m_{\omega'}\int_{A_{n}(k)} {\rm e}^{G_{k}(\zn^{+})} G_{k}(\zn^{+})\,\xi^{2}
\leq
\frac{4\beta^2+\alpha^2}{2\alpha}\,\int_{A_{n}(k)} |\D\xi|^{2}\,G_{k}(\zn^{+})^{2}\,.
$$
Using that $\xi=\sqrt{\vp(\eta)}$ and \eqref{gradxi}, we deduce
\begin{align*}
&\frac{\alpha}{4}\int_{A_{n}(k)}|\D(G_k(\zn^{+})\xi)|^{2}
+
m_{\omega'}\int_{A_{n}(k)} {\rm e}^{G_{k}(\zn^{+})} G_{k}(\zn^{+})\,\vp(\eta)
\\ &\leq
\frac{4\beta^2+\alpha^2}{8\alpha}\norma{\D\eta}{\elle\infty}^2\!\!\!\int_{A_{n}(k)}G_{k}(\zn^{+})^{2}\,\frac{\vp'(\eta)^{2}}{\vp(\eta)}\,.
\end{align*}
Applying Lemma \ref{leo}, with $t=G_k(z_n^+)$, and choosing the constant $C$ as 
$$
C=\frac{4\beta^2+\alpha^2}{4\alpha m_{\omega'}}\,\norma{\D\eta}{\elle\infty}^2\,,
$$ 
we have
\begin{align*}
&\frac{4\beta^2+\alpha^2}{8\alpha}\norma{\D\eta}{\elle\infty}^2\int_{A_{n}(k)}G_{k}(\zn^{+})^{2}\,\frac{\vp'(\eta)^{2}}{\vp(\eta)} \\
&\leq \frac{m_{\omega'}}{2}\int_{A_n(k)}G_k(z_n^+)\,({\rm e}^{G_{k}(\zn^{+})}-1)\,\vp(\eta) + \frac{4\beta^2+\alpha^2}{8\alpha}\norma{\D\eta}{\elle\infty}^2\!\!\!|A_n(k)\,\cap \,\omega'|\,.
\end{align*}
Hence, we obtain
\begin{align*}
&\frac{\alpha}{4}\int_{A_{n}(k)}|\D(G_k(\zn^{+})\xi)|^{2}
+
\frac{m_{\omega'}}{2}\int_{A_{n}(k)} {\rm e}^{G_{k}(\zn^{+})} G_{k}(\zn^{+})\,\vp(\eta) \\
&+
\frac{m_{\omega'}}{2}\, \int_{A_n(k)}G_k(z_n^+)\,\vp(\eta) \leq \frac{4\beta^2+\alpha^2}{8\alpha}\norma{\D\eta}{\elle\infty}^2\!\!\!|A_n(k)\,\cap \,\omega'|\,.
\end{align*}
Dropping the positive terms in the left hand side, we have
\begin{equation*}
\int_{A_{n}(k)}|\D(G_k(\zn^{+})\xi)|^{2} \leq \frac{4\beta^2+\alpha^2}{2\alpha^2}\norma{\D\eta}{\elle\infty}^2\!\!\!|A_n(k)\,\cap \,\omega'|\,.
\end{equation*}
Moreover, denoting with $\mathcal{S}$ the constant given by the Sobolev embedding theorem and recalling that $\xi\equiv 1$ in $\omega''$, we deduce, for $j>k>0$, that
\begin{align*}
&(j-k)^2\,|A_n(j)\,\cap\,\omega''|^\frac{2}{2^*} 
\leq 
\left(\int_{A_n(j)\,\cap\,\omega''}|G_k(z_n^+)|^{2^*}\right)^\frac{2}{2^*} \\
&\leq \left(\int_{A_{n}(k)\,\cap\,\omega'}|G_k(\zn^{+})\xi|^{2^*}\right)^\frac{2}{2^*} 
\leq \mathcal{S}^2\,\frac{4\beta^2+\alpha^2}{2\alpha^2}\norma{\D\eta}{\elle\infty}^2\!\!\!|A_n(k)\,\cap \,\omega'|\,.
\end{align*}
Defining $\dis c_0^\frac{2}{2^*}=\mathcal{S}^2\frac{4\beta^2+\alpha^2}{2\alpha^2}$, we have, for all $\omega''\subset\subset\omega'\subset\subset\omega$, that
\begin{equation}\label{stiminf}
|A_n(j)\,\cap\,\omega''| \leq c_0\frac{\norma{\D\eta}{\elle\infty}^{2^*}\!\!\!|A_n(k)\,\cap \,\omega'|^\frac{2^*}{2}}{(j-k)^{2^*}}\,.
\end{equation} 
Now we consider $\dis R_0=\text{dist}(\omega'',\omega)$. Define 
$$
\omega_r=\{x\in\Omega\,:\,\text{dist}(x,\omega'')<r\}
$$ 
and 
$$
m(k,r)=|A_n(k)\,\cap\,\omega_r|\,,
$$
for every $0<r< R_0$ and $k>0$. Choosing $0\leq r<R < R_0$ and $\eta$ such that $\dis \norma{\D\eta}{\elle\infty}\leq\frac{c_1}{R-r}$ and taking $\omega''=\omega_r$ and $\omega'=\omega_R$ in \eqref{stiminf}, we deduce
\begin{equation*}
m(j,r) \leq c_2\,\frac{m(k,R)^{\frac{2^*}{2}}}{(j-k)^{2^*}(R-r)^{2^*}}\,,
\end{equation*}
where $c_2=c_0\,c_1^{2^*}$. From this inequality it follows, applying Lemma \ref{stamp}, that there exists $M_{\omega'}>0$ (independent on $n$) such that 
$$
\norma{z_n^+}{L^{\infty}(\omega')} \leq M_{\omega'}\,.
$$
Recalling the definition of $\zn$ in terms of $\un$, we therefore have
$$
\un = (n+1)^{\frac{1}{n+1}}{\rm e}^{-\frac{\zn}{n+1}} \geq (n+1)^{\frac{1}{n+1}}{\rm e}^{-\frac{M_{\omega'}}{n+1}}
\quad
\mbox{in $\omega'$,}
$$
which is \eqref{frombelow}.
\end{proof}

We conclude this section with the following remark:

\begin{remark}\label{cons}\sl
As a consequence of estimates \eqref{fromabove} and \eqref{frombelow}, we thus have
$$
\lim_{n \to +\infty}\,\un = 1
\quad
\mbox{uniformly in $\omega'$.}
$$
Repeating this argument for every $\omega'$ contained in $\omega$, we have that $\un$ converges to~1 on $\omega$.
\end{remark}

\section{Proofs of Theorems \ref{main1} and \ref{main2}}
\label{s4}

We start with the proof of Theorem \ref{main1}, in which we recall that $\omega = \{f > 0\}$ is compactly contained in $\Omega$.

\begin{proof}[{\sl Proof of Theorem \ref{main1}}]
We have already proved that
\begin{equation}\label{temp0}
\norma{\un}{\elle\infty} \leq (C(n+1)\norma{f}{\elle\infty})^{\frac{1}{n+1}}\,,
\end{equation}
so that $\un$ is bounded in $\elle\infty$. This implies that there exists $u$ in $\elle\infty$ such that $u_n$ *-weakly converges to $u$ in $\elle\infty$ and, by Remark \ref{cons}, $u\equiv 1$ in $\omega$. We are now going to prove that the right hand side of the equation in \eqref{pbn} is bounded in $\elle1$ uniformly in $n$. As a matter of fact, if $\un$ is the solution of \eqref{pbn}, from Theorem \ref{daBO} and Remark \ref{ben}, it follows that $u_n \in \w$, $u_n \geq c_{\omega,n}>0$ in $\omega$ and $\dis \frac{f}{u_n^n}$ belongs to $\elle\infty$. Then we have, by the results in \cite{sta}, that 
$$
\un(x) = \io G(x,y)\,\frac{f(y)}{\un^{n}(y)}\,dy\,,
\qquad
\forall x \in \Omega\,,
$$
where $G(x,\cdot)$ is the Green function of the linear differential operator defined by the adjoint matrix $M^{*}(x)$ of $M(x)$, i.e., the unique duality solution of
$$
\begin{cases}
-{\rm div}(M^{*}(x)\D G(x,\cdot)) = \delta_{x} & \mbox{in $\Omega$,} \\
\hfill G(x,\cdot) = 0 \hfill & \mbox{on $\partial\Omega$,}
\end{cases}
$$
where $\delta_{x}$ is the Dirac delta concentrated at $x$ in $\Omega$. It is well-known (see for example \cite{LSW}), that for every $\omega' \subset\subset \Omega$ there exists $K > 0$ such that
\begin{equation}\label{dalsw}
G(x,y) \geq \frac{K}{|x-y|^{N-2}}\,,
\qquad
\forall x,\ y \in \omega'\,.
\end{equation}
Fix now $\overline{x}$ in $\Omega \setminus \overline{\omega}$, let $\omega'' \subset\subset \Omega$ be such that $\omega \subset \omega''$ and $\overline{x}$ belongs to $\omega''$, and let $K$ be such that \eqref{dalsw} holds. We then have
\begin{align*}
(C(n+1)\norma{f}{\elle\infty})^{\frac{1}{n+1}} \geq\un(\overline{x}) &= \io G(\overline{x},y)\,\frac{f(y)}{\un^{n}(y)}\,dy \\
&\geq \io \frac{K}{|\overline{x}-y|^{N-2}}\,\frac{f(y)}{\un^{n}(y)}\,dy
\\
&\geq \frac{K}{{\rm{diam}}(\Omega)^{N-2}} \int_{\omega}\frac{f(y)}{\un^{n}(y)}\,dy\,.
\end{align*}
Therefore, there exists $M > 0$ such that
\begin{equation}\label{temp1}
\int_{\omega}\frac{f(x)}{\un^{n}} = \io\frac{f(x)}{\un^{n}} \leq M\,,
\end{equation}
i.e., the right hand side of the equation in \eqref{pbn} is bounded in $\elle1$. Observe now that for every $\omega' \subset\subset \omega$ there exists $M_{\omega'}$ such that
$$
\un(x) \geq (n+1)^{\frac{1}{n+1}}\,{\rm e}^{-\frac{M_{\omega'}}{n+1}},
\quad
\mbox{in $\omega'$.}
$$
Therefore,
$$
\int_{\omega'}\frac{f(x)}{\un^{n}}
\leq
\frac{|\omega'|{\rm e}^{\frac{n M_{\omega'}}{n+1}}\norma{f}{\elle\infty}}{(n+1)^{\frac{n}{n+1}}},
$$
so that
\begin{equation}\label{tozero}
\lim_{n \to +\infty}\,\int_{\omega'}\frac{f(x)}{\un^{n}} = 0,
\end{equation}
i.e., the right hand side converges to zero in $L^{1}_{{\rm loc}}(\omega)$. Let now $\mu$ be the bounded Radon measure such that
$$
\frac{f(x)}{\un^{n}}
\to
\mu,
\quad
\mbox{in the $*$-weak topology of measures.}
$$
Clearly, by the assumption on $f$, $\mu\LL (\Omega \setminus \overline{\omega}) = 0$, and, by \eqref{tozero}, $\mu \LL \omega = 0$, so that $\mu = \mu \LL \partial\omega$. Moreover, by Remark \ref{ben}, we can take $u_n$ as test function in \eqref{pbn} and we obtain, using \eqref{albe}, \eqref{temp0} and \eqref{temp1}, that
$$
\io |\D u_n|^2 \leq \io \frac{f(x)\,u_n}{u_n^n} \leq \norma{u_n}{\elle\infty}\!\!\!\io\frac{f(x)}{u_n^n} \leq C\,,
$$
then $u_n$ weakly converges to $u$ in $\w$ as $n$ tends to infinity. Recalling that, by Remark \ref{ben}, $u_n$ is the (unique) weak solution of \eqref{pbn}, that is 
\begin{equation}\label{testpn}
\io M(x)\D u_n \cdot \D\vp = \io \frac{f\,\vp}{u_n^{n}},
\qquad
\forall \vp \in W^{1,2}_{0}(\Omega)\,,
\end{equation}
we obtain, letting $n$ tend to infinity, that 
\begin{equation}\label{testp}
\io M(x)\D u\cdot \D\vp = \io \vp\,d\mu\,,
\qquad
\forall \vp \in C^1_{0}(\Omega)\,,
\end{equation}
so that $u$ is a distributional solution with finite energy of the limit problem \eqref{pblim0}.
\end{proof}

\begin{remark}\label{dual}\sl
We observe that $u_n$ is also the unique duality solution of \eqref{pbn}, i.e. 
\begin{equation}\label{temp2}
\io u_n\,g=\io \frac{f}{u_n^n}\,v\,, \qquad \forall g\in L^\infty(\Omega)\,,
\end{equation}
where $v\in \w\cap\elle\infty$ is the unique weak solution of
\begin{equation}\label{star}
\begin{cases}
-{\rm div}(M^*(x)\nabla v) = g & \mbox{in $\Omega$,} \\
\hfill v = 0 \hfill & \mbox{on $\partial\Omega$.}
\end{cases}
\end{equation}
This implies, letting $n$ tend to infinity in \eqref{temp2} and using the standard results contained in \cite{sta}, that $u$ is the unique duality solution of \eqref{pblim0}.
\end{remark}

Now we prove Theorem \ref{main2}. Here let us recall that for every $\omega \subset\subset \Omega$ there exists $c_{\omega} > 0$ such that $f \geq c_{\omega}$ in $\omega$ and that $\{\omega_{n}\}$ is an increasing sequence of compactly contained subsets of $\Omega$ such that their union is $\Omega$.

\begin{proof}[{\sl Proof of Theorem \ref{main2}}]
Let $u_n$ be the solution of \eqref{pbn2}. It follows, from the fact that $f(x)\,\chi_{\omega_n}(x)$ has compact support in $\Omega$ and using Remark \ref{ben}, that $u_n$ belongs to $W^{1,2}_{0}(\Omega)$ and $\dis \frac{f(x)\,\chi_{\omega_n}(x)}{u_n^n}$ belongs to $L^1(\Omega)$. Once again as a consequence of Theorem \ref{otherprop} we have that $\{u_n\}$ is bounded in $L^\infty(\Omega)$. Then there exists $u$ in $\elle\infty$ such that $u_n$ *-weakly converges to $u$ in $L^{\infty}(\Omega)$. Moreover, by Remark \ref{cons}, we deduce that $u_n$ uniformly converges to $1$ in $\omega$, for every $\omega\subset\subset \Omega$, hence $u\equiv 1$ in $\Omega$. If we assume that the sequence $\dis \left\{\frac{f(x)\,\chi_{\omega_n}(x)}{u_n^n}\right\}$ is bounded in $L^1(\Omega)$, then it *-weakly converges to $\mu$ in the topology of measure. Repeating the same arguments contained in Remark \ref{dual} we obtain 
$$
\io u\, g=\io v\,d\mu\,, \qquad \forall g \in \elle\infty\,,
$$
where $v$ in $\w$ is the weak solution of \eqref{star}. Then $u$ in $L^\infty(\Omega)$ is the duality solution of \eqref{pblim0}, so that $u$ belongs to $W^{1,1}_{0}(\Omega)$. Since $u\equiv 1$ in $\Omega$, there is a contradiction. Hence, the right hand side of \eqref{pbn2} is not bounded in $L^1(\Omega)$ and there cannot be any limit equation.
\end{proof}

\section{One-dimensional solutions and Proof of Theorem \ref{main3}}
\label{s5}

First we prove a result that makes the link between a distributional solution of \eqref{pbn} and a distributional solution with finite energy of \eqref{parn} rigorous.

\begin{prop}
\label{link}
Let $f$ be a nonnegative function belonging to $\linf$. If $u_n$ is a solution of \eqref{pbn} given by Theorem \ref{daBO}, then $\dis v_n=\frac{u_n^{n+1}}{n+1}$ is a distributional solution of \eqref{parn} with finite energy.
\end{prop}

\begin{proof}
We already know, by Theorem \ref{otherprop}, that $\dis u_n^{n+1}$ belongs to $\w$, so that $v_n$ belongs to $\w$. With the same argument we have that $u_n^n$ belongs to $\w$. Let $\vp$ be a function in $C^1_c(\Omega)$, we have that $u_n^n\vp$ is a function in $\w$ with compact support ($\omega=\supp(\vp)$). Then we can take $u_n^n\vp$ as test function in \eqref{defsol} and we obtain that
\begin{equation}
\label{pr41}
\io\nabla u_n\cdot\nabla \vp \,u_n^n+n\io \nabla u_n\cdot\nabla u_n \,u_n^{n-1}\vp=\io f\vp.
\end{equation}
If we rewrite \eqref{pr41}, using that $u_n\geq c_{\omega,n}$ in $\omega$, we have
$$
\io\nabla\left(\frac{u_n^{n+1}}{n+1}\right)\cdot\nabla \vp+n\io |\nabla u_n|^2\,\frac{u_n^{2n}}{u_n^{n+1}}\vp=\io f\vp.
$$
Hence, by definition of $v_n$, we deduce that
$$
\io\nabla v_n\cdot\nabla \vp+\frac{n}{n+1}\io \frac{|\nabla v_n|^2}{v_n}\vp=\io f\vp,
$$
that is $v_n$ is a distributional solution with finite energy of \eqref{parn}.
\end{proof}

\begin{remark}
\label{remcomsup}
We note that for every $\omega \subset\subset\Omega$ we know, by Theorem \ref{daBO}, that $u_n\geq c_{\omega,n}$ in $\omega$. Then $\dis v_n\geq \frac{c_{\omega,n}^{n+1}}{n+1}$ in $\omega$. Using this property and that $v_n$ has finite energy we can extend the class of test functions for \eqref{parn} from $C^1_c(\Omega)$ to $\w$ with compact support.
\end{remark}

Now we study \eqref{pbn} in the one-dimensional case to better understand what happens, if $f$ is {\sl strictly positive}, to $u_n$ and to the related $v_n$ by passing to the limit as $n$ tends to infinity. \\

Fix $n > 3$ in $\N$. We consider \eqref{pbn} with $\Omega=(-R,R)$, $R>0$, $M(x)\equiv I$ and $f\equiv 1$ in $(-R,R)$. So that we have
\begin{equation}
\label{oned1}
\begin{cases}
\dis -u_n'' = \frac{1}{u_n^{n}} & \mbox{in $(-R,R)$,} \\
\hfill u_n(\pm R) = 0. \\
\end{cases}
\end{equation}
In order to study \eqref{oned1} we focus on the solutions $y_n$ of the following Cauchy problem 
\begin{equation}
\label{oned2}
\begin{cases}
\dis -y_n''(t) = \frac{1}{y_n^{n}(t)} & \mbox{for $t\geq 0$,} \\
y_n(0) = \alpha_n, \\
y_n'(0)=0,
\end{cases}
\end{equation}
where $\alpha_n$ is a positive real number that we will choose later. Defining $\dis w_n=\frac{y_n}{\alpha_n}$, we can rewrite \eqref{oned2} as
\begin{equation}
\label{oned3}
\begin{cases}
\dis -w_n''(t) = \frac{1}{\alpha_n^{n+1} w_n^{n}(t)} & \mbox{for $t\geq 0$,} \\
w_n(0) =1, \\
w_n'(0)=0.
\end{cases}
\end{equation}
Since $\dis \frac{1}{\alpha_n^{n+1}s^n}$ is Lipschitz continuous near $s=1$, there exists a unique solution $w_n$ locally near $t=0$. It is easy, by a classical iteration argument, to extend the definition interval of $w_n$ to $[0,T_n)$, where $T_n<+\infty$ is the first zero of $w_n$ (i.e. $w_n(T_n)=0$) when it occurs, otherwise $T_n=+\infty$. Hence $w_n$ is concave ($w_n''(t)<0$), decreasing ($w_n'(t)< 0$) and $0<w_n(t)\leq 1$ for $t\in[0,T_n)$ and it belongs to $C^\infty((0,T_n))$. \\
Now multiplying the equation by $w_n'(t)$ we have
$$
-\frac{[w_n'(t)^2]'}{2}=\frac{w_n'(t)}{\alpha_n^{n+1} w_n^{n}(t)},
$$
hence, integrating on $[0,s]$, with $0<s<T_n$, and recalling that $w_n'(0)=0$, we have
$$
w_n'(s)^2=\frac{2}{(n-1)\alpha_n^{n+1}}(w_n^{1-n}(s)-1).
$$
Since $w_n'(s)< 0$ we deduce
\begin{equation}
\label{oned4}
w_n'(s)=-\sqrt{\frac{2}{(n-1)\alpha_n^{n+1}}}\,(w_n^{1-n}(s)-1)^{\frac{1}{2}},
\end{equation}
therefore we can divide \eqref{oned4} by $\dis (w_n^{1-n}(s)-1)^{\frac{1}{2}}$ and integrate on $[0,t]$, with $0\leq t<T_n$, to obtain 
\begin{equation}
\label{oned5}
\int_0^t \frac{w_n'(s)}{(w_n^{1-n}(s)-1)^{\frac{1}{2}}}\,ds=-\sqrt{\frac{2}{(n-1)\alpha_n^{n+1}}}\,\,t.
\end{equation}
Setting $r=w_n(s)$ in the first integral of \eqref{oned5} and recalling that $w_n(0)=1$, we have
$$
\int_{w_n(t)}^1 \frac{r^{\frac{n-1}{2}}}{(1-r^{n-1})^\frac{1}{2}}\,dr=\sqrt{\frac{2}{(n-1)\alpha_n^{n+1}}}\,\,t.
$$
Once again we can perform the change of variable $h=1-r^{n-1}$ to deduce
\begin{equation}
\label{oned6}
\int_0^{1-w_n^{n-1}(t)}\frac{1}{h^{\frac{1}{2}}\,(1-h)^{\frac{n-3}{2(n-1)}}}\,\,dh=\sqrt{\frac{2(n-1)}{\alpha_n^{n+1}}}\,t.
\end{equation}
Define $\dis I_n(t):=\int_0^{1-w_n^{n-1}(t)}\frac{1}{h^{\frac{1}{2}}(1-h)^{\frac{n-3}{2(n-1)}}}\,\,dh$ for $t\geq 0$, then $I_n(0)=0$ and $I_n$ is a continuous positive and increasing function in $[0,T_n)$, so that $I_n(t)\leq I_n(T_n)$. Thanks to the results in \cite{AS} we obtain
\begin{equation}
\label{oned7}
I_n(T_n)=\int_0^{1}\frac{1}{h^{\frac{1}{2}}(1-h)^{\frac{n-3}{2(n-1)}}}\,\,dh=\sqrt{\pi}\frac{\Gamma\left(\frac{1}{2}+\frac{1}{n-1}\right)}{\Gamma\left(\frac{n}{n-1}\right)},
\end{equation}
where $\Gamma(s)$ is defined in \eqref{gammadef}. Thus we can extend $I_n(t)$ in $[0,T_n]$ and it is uniformly bounded for every $n\in\N$ and $t\in [0,T_n]$. Moreover, from \eqref{oned7} and computing \eqref{oned6} for $t=T_n$, we have
\begin{equation}
\label{oned8}
T_n=\sqrt\frac{\pi\,\alpha_n^{n+1}}{2(n-1)}\frac{\Gamma\left(\frac{1}{2}+\frac{1}{n-1}\right)}{\Gamma\left(\frac{n}{n-1}\right)}.
\end{equation}
We observe that $T_n$ and $\alpha_n$ are such that if $\alpha_n$ tends to infinity also $T_n$ tends to infinity. Recalling that we want a solution for \eqref{oned1} that is zero if $t=R$, imposing $T_n=R$ for every $n$ in $\N$ we find that
\begin{equation}
\label{oned9}
\alpha_n=\left(\frac{2R^2\,(n-1)\,\Gamma^2\left(\frac{n}{n-1}\right)}{\pi\,\Gamma^2\left(\frac{1}{2}+\frac{1}{n-1}\right)}\right)^{\frac{1}{n+1}}.
\end{equation}
Hence, with this value of $\alpha_n$, $w_n(R)=0$ for every $n$ in $\N$ and $w_n$ belongs to $C^2((0,R))$. Thanks to the initial condition $w_n'(0)=0$, we can extend $w_n$ to an even function $\tilde{w}_n$ on $[-R,R]$ in the following way
$$
\tilde{w}_n(t)=
\begin{cases}
w_n(t) & \mbox{for $t\in [0,R]$} \\
w_n(-t) & \mbox{for $t\in [-R,0)$}\,.
\end{cases}
$$
So $\tilde{w}_n$ belongs to $C^2_0((-R,R))$ and is the classical solution of 
\begin{equation}
\label{onedapp}
\begin{cases}
\dis -\tilde{w}_n''(t) = \frac{1}{\alpha_n^{n+1} \tilde{w}_n^{n}(t)} & \mbox{for $t\geq 0$,} \\
\tilde{w}_n(\pm R) =0\,. 
\end{cases}
\end{equation}
Setting $u_n(t)=\alpha_n\,\tilde{w}_n(t)$ for $t$ in $[-R,R]$ we have that $u_n$ belongs to $C^2_0((-R,R))$ and is the classical solution of \eqref{oned1}. This implies that $\dis v_n(t)=\frac{u_n(t)^{n+1}}{n+1}$ is a classical solution (in $C^2_0((-R,R))$) of 
\begin{equation}
\label{oned13}
\begin{cases}
\dis -v_n'' \,+\,\frac{n}{n+1} \frac{|v_n'|^2}{v_n}\,=\,1 & \mbox{in $(-R,R)$,} \\
v_n(\pm R) = 0, \\
\end{cases}
\end{equation}
that is \eqref{parn} in the one-dimensional case. Multiplying the equation \eqref{oned13} by $v_n$ and integrating by parts on $(-R,R)$ we obtain that $\{v_n\}$ is bounded in $W_0^{1,2}((-R,R))$. By definition of $v_n$, this implies that $\{\tilde{w}^{n+1}_n\}$ is bounded in $W_0^{1,2}((-R,R))$. Using the Rellich-Kondrachov theorem we deduce that there exist a subsequence, still indexed by $\tilde{w}^{n+1}_n$, and a function $g:(-R,R)\to [0,1]$ in $C_0((-R,R))$ such that $\tilde{w}^{n+1}_n$ uniformly converges to $g$ in $(-R,R)$. We want to make $g$ explicit. \\
By definition of $\tilde{w}_n$ it follows that 
$$
\lim_{n\to\infty}w_n^{n-1}(t)=\lim_{n\to\infty}\left(w_n^{n+1}(t)\right)^{\frac{n-1}{n+1}}=g(t)\,,
$$
uniformly in $(0,R)$. Combining \eqref{oned6} and \eqref{oned9} we obtain
\begin{equation}
\label{oned10}
\int_0^{1-w_n^{n-1}(t)}\frac{1}{h^{\frac{1}{2}}\,(1-h)^{\frac{n-3}{2(n-1)}}}\,\,dh=\frac{\sqrt{\pi}}{R}\frac{\Gamma\left(\frac{1}{2}+\frac{1}{n-1}\right)}{\Gamma\left(\frac{n}{n-1}\right)}\,t.
\end{equation}
Passing to the limit in \eqref{oned10} as $n$ tends to infinity we obtain the explicit expression of $g$. Indeed we have, by Lebesgue theorem and from well known result of integral calculus, that 
$$
2\arcsin(\sqrt{1-g(t)})=\lim_{n\to\infty}\int_0^{1-w_n^{n-1}(t)}\frac{1}{h^{\frac{1}{2}}\,(1-h)^{\frac{n-3}{2(n-1)}}}\,\,dh=\lim_{n\to\infty}\frac{\sqrt{\pi}}{R}\frac{\Gamma\left(\frac{1}{2}+\frac{1}{n-1}\right)}{\Gamma\left(\frac{n}{n-1}\right)}\,t=\frac{\pi}{R}\,t.
$$
It follows that
$$
g(t)=1-\sin^2\left(\frac{\pi}{2R}\,t\right)=\cos^2\left(\frac{\pi}{2R}\,t\right).
$$
So $g$ is an even $C^\infty$ function defined on $\R$, in particular on $[-R,R]$. \\
Fix now $t$ in $(-R,R)$. We want to prove that $\tilde{w}_n(t)$ tends to $1$ as $n$ tends to infinity.
We assume, by contradiction, that 
$$
\lim_{n\to\infty}\tilde{w}_n(t)=\beta<1.
$$
Defining $\dis \varepsilon:=\frac{1-\beta}{2}$, we deduce, for $n$ large enough, that $\tilde{w}_n(t)\leq 1-\varepsilon$. So that
$$
\tilde{w}_n^{n+1}(t)\leq (1-\varepsilon)^{n+1},
$$
and, letting $n$ tend to infinity, we obtain $\dis \cos^2\left(\frac{\pi}{2R}\,t\right)=0$. Since $t\neq \pm R$, we find a contradiction, then $\tilde{w}_n(t)$ tends to $1$, as $n$ tends to infinity, for every $t$ in $(-R,R)$.\\
Now we return to problem \eqref{oned1} recalling that $u_n(t)=\alpha_n\,\tilde{w}_n(t)$. From \eqref{oned9} and using that $\tilde{w}_n(t)$ tends to $1$, as $n$ tends to infinity, for $t$ in $(-R,R)$, it follows that
$$
\lim_{n\to\infty}u_n(t)=1, \qquad \forall t\in(-R,R).
$$
This result is exactly the one-dimensional version of Remark \ref{cons}. From \eqref{oned9}, we deduce that
$$
v_n(t)=\frac{2R^2\,(n-1)\,\Gamma^2\left(\frac{n}{n-1}\right)}{\pi\,(n+1)\,\Gamma^2\left(\frac{1}{2}+\frac{1}{n-1}\right)}\,\tilde{w}_n^{n+1}(t),
$$
so that we have that there exists a limit function $v:[-R,R]\to\R$ such that
$$
v(t)=\lim_{n\to\infty}v_n(t)=\frac{2R^2}{\pi^2}\, \cos^2\left(\frac{\pi}{2R}\,t\right).
$$
After a little algebra we obtain that $v$ is a classical solution of 
\begin{equation*}
\begin{cases}
\dis -v'' \,+\, \frac{|v'|^2}{v}\,=\,1 & \mbox{in $(-R,R)$,} \\
v_n(\pm R) = 0, \\
\end{cases}
\end{equation*}
that is \eqref{pbv4}. Thus we have proved Theorem \ref{main3} in the one-dimensional case. \\
Finally we prove Theorem \ref{main3} in the $N$-dimensional case, here we recall that $f$ is {\sl strictly positive}.

\begin{proof}[{\sl Proof of Theorem \ref{main3}}]
Let $u_n$ be the solution of \eqref{pbn} given by Theorem \ref{daBO}. It follows from Proposition \ref{link} that $v_n$ are distributional solutions of \eqref{parn}. \\
By assumption for every $\omega \subset\subset \Omega$ there exists a positive constant $c_\omega$ such that $f\geq c_\omega$. This implies, by Theorem \ref{dasotto}, that
$$
\un\, \geq\, (n+1)^{\frac{1}{n+1}}{\rm e}^{-\frac{M_{\omega}}{n+1}},
$$ 
then 
\begin{equation}
\label{sottov}
v_n\, \geq \, {\rm e}^{-M_{\omega}}, \qquad \forall \,\omega\subset\subset\Omega,
\end{equation}
with $M_{\omega}$ a positive constant depending only on $\omega$. So $v_n$ is locally uniformly positive. Moreover, by Theorem \ref{otherprop}, we have that $v_n$ belongs to $\w$ and
$$
\dis \norma{v_n}{\linf}\leq C\norma{f}{\linf}\!\!,
$$
where $C$ is a positive constant. \\
Choosing a nonnegative $\varphi$ belonging to $C_c^1(\Omega)$ as test function in \eqref{parn} and dropping the nonnegative integral involving the quadratic gradient term, we deduce that
\begin{equation}
\label{pos1}
\io \nabla v_n\cdot \nabla \varphi\leq\io f\,\varphi\,.
\end{equation}
As a consequence of the density of $C_c^1(\Omega)$ in $\w$ we can extend \eqref{pos1} for every nonnegative $\varphi$ in $\w$. Choosing $v_n$ as test function and using H\"{o}lder's inequality and the Sobolev embedding theorem, we obtain
$$
\io |\nabla v_n|^2 \leq \io f\,v_n\leq \norma{f}{L^{\frac{2N}{N+2}}(\Omega)}\norma{v_n}{L^{2^*}(\Omega)} \leq \mathcal{S} \norma{f}{L^{\frac{2N}{N+2}}(\Omega)}\norma{v_n}{\w}\!\!,
$$
where $\mathcal{S}$ is the Sobolev constant. Hence $\{v_n\}$ is bounded in $\w$. Thus, up to a subsequence, it follows that there exists $v$ belonging to $\w\cap\linf$ such that
\begin{equation}
\label{pos2}
\begin{array}{l}
v_{n} \rightarrow v \text{ weakly in } \w \text{ and weakly-* in } \linf,\\
v_{n} \rightarrow v \text{ strongly in } L^q(\Omega),\,\forall\,q<+\infty, \text{ and a.e. in } \Omega.
\end{array}
\end{equation}
In order to pass to the limit in \eqref{parn} we first prove that $v_n$ strongly converges to $v$ in $W^{1,2}_{{\rm loc}}(\Omega)$, that is
\begin{equation}
\label{pos3}
\lim_{n\to +\infty}\io |\nabla(v_n-v)|^2\varphi=0, \qquad \forall\,\varphi\in C^1_c(\Omega) \,\text{ with }\, \varphi \geq 0\,.
\end{equation}
We consider the function $\phi_\lambda(s)$ defined in \eqref{defphilam} and, choosing $\phi_\lambda(v_n-v)\varphi$ as test function in \eqref{parn}, we obtain 
\begin{align*}
&\io \nabla v_n\cdot\nabla(v_n-v)\,\phi'_\lambda(v_n-v)\,\varphi\,+\,\io\nabla v_n\cdot \nabla\varphi\,\phi_\lambda(v_n-v)\, \\
&+\,\frac{n}{n+1}\io \frac{|\nabla v_n|^2}{v_n}\,\phi_\lambda(v_n-v)\varphi\,=\,\io f\,\,\phi_\lambda(v_n-v)\,\varphi. \nonumber
\end{align*}
It follows from \eqref{pos2} and using Lebesgue theorem that 
$$
\lim_{n\to+\infty}\io\nabla v_n\cdot \nabla\varphi\,\phi_\lambda(v_n-v)=0 \quad\text{and}\quad
\lim_{n\to+\infty}\io f\,\,\phi_\lambda(v_n-v)\,\varphi=0.
$$
Thus 
\begin{equation}
\label{pos4}
\io \nabla v_n\cdot\nabla(v_n-v)\,\phi'_\lambda(v_n-v)\,\varphi\,+\,\frac{n}{n+1}\io \frac{|\nabla v_n|^2}{v_n}\,\phi_\lambda(v_n-v)\varphi\,=\,\epsilon(n).
\end{equation}
Moreover, setting $\omega_\varphi=\supp(\varphi)$ and using \eqref{sottov}, we deduce that
\begin{align*}
\frac{n}{n+1}\io \frac{|\nabla v_n|^2}{v_n}\,\phi_\lambda(v_n-v)\varphi &\geq \, -\frac{n}{n+1}\io \frac{|\nabla v_n|^2}{v_n}\,|\phi_\lambda(v_n-v)|\varphi \\
&\geq -\,{\rm e}^{M_{\omega_\varphi}} \io |\nabla v_n|^2\,|\phi_\lambda(v_n-v)|\,\varphi,
\end{align*}
so that
\begin{equation}
\label{pos5}
\io \nabla v_n\cdot\nabla(v_n-v)\,\phi'_\lambda(v_n-v)\,\varphi\,-\,{\rm e}^{M_{\omega_\varphi}} \io |\nabla v_n|^2\,|\phi_\lambda(v_n-v)|\,\varphi=\epsilon(n).
\end{equation}
We can add the following term to \eqref{pos5} 
$$
-\io \nabla v\cdot \nabla(v_n-v)\,\phi'_\lambda(v_n-v)\,\varphi
$$
and, noting that by \eqref{pos2} this quantity tends to $0$ letting $n$ go to infinity, we obtain 
\begin{equation}
\label{pos6}
\io |\nabla(v_n-v)|^2\,\phi'_\lambda(v_n-v)\,\varphi\,-\,{\rm e}^{M_{\omega_\varphi}} \io |\nabla v_n|^2\,|\phi_\lambda(v_n-v)|\,\varphi=\epsilon(n).
\end{equation}
Since using once again \eqref{pos2} we have
\begin{align*}
&\io |\nabla v_n|^2\,|\phi_\lambda(v_n-v)|\,\varphi\leq 2\io |\nabla(v_n-v)|^2\,|\phi_\lambda(v_n-v)|\,\varphi\, \\
&+\,2\io |\nabla v|^2\,|\phi_\lambda(v_n-v)|\,\varphi\,=\,2\io |\nabla(v_n-v)|^2\,|\phi_\lambda(v_n-v)|\,\varphi\,+\,\epsilon(n)\,,
\end{align*}
we deduce that
$$
\io |\nabla(v_n-v)|^2\,\left\{\phi'_\lambda(v_n-v)-2{\rm e}^{M_{\omega_\varphi}}|\phi_\lambda(v_n-v)|\right\}\,\varphi\,=\,\epsilon(n).
$$
Choosing $\lambda\geq {\rm e}^{2M_{\omega_\varphi}}$, thanks to \eqref{defphilam2} we have that  $\dis \{\phi'_\lambda(v_n-v)-2{\rm e}^{M_{\omega_\varphi}}|\phi_\lambda(v_n-v)|\}\geq \frac{1}{2}$, hence \eqref{pos3} holds and 
\begin{equation}
\label{pos7}
v_n \rightarrow v \text{ strongly in } W^{1,2}_{{\rm loc}}(\Omega). 
\end{equation}
Now we pass to the limit in \eqref{parn} with test functions $\varphi$ belonging to $\w\cap\linf$ with compact support. We have, by \eqref{pos3}, that
$$
\lim_{n\to+\infty}\io \nabla v_n\cdot \nabla \phi=\io\nabla v\cdot\nabla \varphi,
$$
and, using \eqref{pos7}, \eqref{sottov} with $\omega=\supp(\varphi)$ and Lebesgue theorem, we deduce
$$
\lim_{n\to+\infty}\frac{n}{n+1}\io \frac{|\nabla v_n|^2}{v_n}\,\varphi\,=\,\io \frac{|\nabla v|^2}{v}\,\varphi,
$$
so that
\begin{equation}
\label{pos8}
\io\nabla v\cdot\nabla \varphi\,+\,\io \frac{|\nabla v|^2}{v}\,\varphi\,=\,\io f\,\varphi\,,
\end{equation}
for all $\varphi$ in $\w\cap\linf$ with compact support. \\
Let $\varphi$ be a nonnegative function in $\w\cap\linf$. Let $\{\varphi_m\}$ in $C^1_c(\Omega)$ be a sequence of nonnegative functions that converges to $\vp$ strongly in $\w$. Taking $\varphi_m\wedge\varphi$, which belongs to $\w\cap\linf$ with compact support, as test function in \eqref{pos8}, we obtain
\begin{equation}
\label{pos9}
\io \frac{|\nabla v|^2}{v}\,(\varphi_m\wedge\varphi)\,= \,\io f\,(\varphi_m\wedge\varphi)\,-\,\io\nabla v\cdot\nabla(\varphi_m\wedge\varphi)\,.
\end{equation}
Since $\varphi_m\wedge\varphi$ strongly converges to $\varphi$ in $\w$ we have
\begin{equation}
\label{pos10}
\lim_{m\to+\infty}\io\left\{f\,(\varphi_m\wedge\varphi)-\io\nabla v\cdot\nabla(\varphi_m\wedge\varphi)\right\}=\,\io f\,\varphi\,-\,\io\nabla v\cdot\nabla \varphi\,.
\end{equation}
Moreover $\dis\frac{|\nabla v|^2}{v}\,(\varphi_m\wedge\varphi)$ is a nonnegative function that converges to $\dis\frac{|\nabla v|^2}{v}\,\varphi$ almost everywhere in $\Omega$. Applying Fatou's lemma on the left hand side of \eqref{pos9} and using \eqref{pos10} we deduce that
$$
\io\frac{|\nabla v|^2}{v}\,\varphi\,\leq\,\liminf_{m\to+\infty}\,\io\frac{|\nabla v|^2}{v}\,(\varphi_m\wedge\varphi)\,=\,\io f\,\varphi\,-\,\io\nabla v\cdot\nabla \varphi\,,
$$
so that $\dis\frac{|\nabla v|^2}{v}\,\varphi$ belongs to $\luno$. Since $\dis \frac{|\nabla v|^2}{v}\,(\varphi_m\wedge\varphi)\leq \frac{|\nabla v|^2}{v}\,\varphi$, by Lebesgue theorem, we have
\begin{equation}
\label{pos11}
\lim_{m\to+\infty}\io \frac{|\nabla v|^2}{v}\,(\varphi_m\wedge\varphi)\,= \,\io\frac{|\nabla v|^2}{v}\,\varphi\,.
\end{equation}
As a consequence of \eqref{pos10} and \eqref{pos11} we obtain 
\begin{equation}
\label{pos12}
\io\nabla v\cdot\nabla \varphi\,+\,\io \frac{|\nabla v|^2}{v}\,\varphi\,=\,\io f\,\varphi\,,\quad\forall\,\varphi\,\geq 0\text{ in } \w\cap\linf\,.
\end{equation}
Furthermore, taking $\dis \frac{T_\varepsilon(v)}{\varepsilon}$ as test function in \eqref{pos12} and dropping a positive term, we deduce
\begin{equation}
\label{pos13}
\io\frac{|\nabla v|^2}{v}\,\frac{T_\varepsilon(v)}{\varepsilon}\,\leq\,\io f\,\frac{T_\varepsilon(v)}{\varepsilon}\,.
\end{equation}
Applying Fatou's lemma on the left hand side of \eqref{pos13} and noting that $T_\varepsilon(v)\leq\varepsilon$ we have
$$
\io\frac{|\nabla v|^2}{v}\,\leq\,\liminf_{\varepsilon\to 0}\io\frac{|\nabla v|^2}{v}\,\frac{T_\varepsilon(v)}{\varepsilon}\,\leq\,\io f,
$$
so $\dis \frac{|\nabla v|^2}{v}$ belongs to $\luno$. Since we can write each $\vp\in W^{1,2}_0(\Omega)\cap L^\infty(\Omega)$ as the difference between its positive and its negative part, we trivially deduce that \eqref{pos12} holds for all $\vp\in W^{1,2}_0(\Omega)\cap L^\infty(\Omega)$, so that $v$ is a weak solution of \eqref{pbv4}.
\end{proof}

\begin{remark}
We note that we can also consider test functions only belonging to $\w$ in \eqref{pos12}. Indeed let $\vp$ be in $\w$, then $T_k(\vp^+)$ is a positive function belonging to $\w\cap\linf$ that strongly converges to $\vp^+$ in $\w$ as $k$ tends to infinity. Taking $T_k(\vp^+)$ as test function in \eqref{pos12} and letting $k$ tend to infinity, by Lebesgue theorem and Beppo Levi theorem, we deduce
\begin{equation}
\label{pos14}
\io\nabla v\cdot\nabla \varphi^+\,+\,\io \frac{|\nabla v|^2}{v}\,\varphi^+\,=\,\io f\,\varphi^+\,.
\end{equation}
In the same way we obtain
\begin{equation}
\label{pos15}
\io\nabla v\cdot\nabla \varphi^-\,+\,\io \frac{|\nabla v|^2}{v}\,\varphi^-\,=\,\io f\,\varphi^-\,,
\end{equation}
so that subtracting \eqref{pos15} to \eqref{pos14} we have that \eqref{pos12} holds for every $\vp$ belonging to $\w$.
\end{remark}

\begin{remark}
\label{vnbuo}
To prove that $\{v_n\}$ is bounded in $\w$ and \eqref{pos2} we only used that $f$ is nonnegative and belongs to $\linf$.
\end{remark}

\section{Proof of Theorem \ref{main4}}
\label{s6}

Here we prove Theorem \ref{main4}. We fix $n > 3$ in $\N$.

\begin{proof}[{\sl Proof of Theorem \ref{main4}}]
First we study the behavior of the weak solution of \eqref{oned00} given by Theorem \ref{daBO}. In order to study $u_n$ we use the construction of one-dimensional solutions done in the previous section, in which we have proved that there exists a function $w_n$ in $C^2((0,T_n))$ classical solution of
\begin{equation}
\label{oned01}
\begin{cases}
\dis -w_n''(t) = \frac{1}{\alpha_n^{n+1} w_n^{n}(t)} & \mbox{ in }(0,T_n), \\
w_n(0) =1, \\
w_n'(0)=0,
\end{cases}
\end{equation}
where $T_n$ is the first zero of $w_n$. We recall that $0<w_n(t)<1$, $w_n$ is concave ($w''_n(t)<0$) and decreasing ($w'_n(t)<0$) for every $t$ in $(0,T_n)$. Moreover we have obtained that 
\begin{equation}
\label{oned02}
w_n'(t)=-\sqrt{\frac{2}{(n-1)\alpha_n^{n+1}}}\,(w_n^{1-n}(t)-1)^{\frac{1}{2}},
\end{equation}
and, by integrating, that
\begin{equation}
\label{oned03}
S_n(1-w_n^{n-1}(t)):=\int_0^{1-w_n^{n-1}(t)}\frac{1}{h^{\frac{1}{2}}\,(1-h)^{\frac{n-3}{2(n-1)}}}\,\,dh=\sqrt{\frac{2(n-1)}{\alpha_n^{n+1}}}\,t\,,
\end{equation}
for every $t$ in $[0,T_n)$. So that $S_n:[0,1)\to[0,S_n(1))$ is a nonnegative, continuous and strictly increasing function. Recalling \eqref{oned7} we have that $S_n(1)=I_n(T_n)$, that is uniformly bounded, thus we can extend $S_n$ in $1$ to have $S_n:[0,1]\to[0,S_n(1)]$. Then there exists the inverse function $S_n^{-1}:[0,S_n(1)]\to [0,1]$. Furthermore we recall that
\begin{equation}
\label{oned04}
T_n=\sqrt\frac{\pi\,\alpha_n^{n+1}}{2(n-1)}\frac{\Gamma\left(\frac{1}{2}+\frac{1}{n-1}\right)}{\Gamma\left(\frac{n}{n-1}\right)}.
\end{equation}
In order to have $1<T_n<+\infty$ for every $n$ we can choose $\dis \alpha_n=(c_n(n-1))^{\frac{1}{n+1}}$, with $c_n$ a positive constant such that 
\begin{equation}
\label{oned05}
c_n>\frac{2\,\Gamma^2\left(\frac{n}{n-1}\right)}{\pi\,\Gamma^2\left(\frac{1}{2}+\frac{1}{n-1}\right)}=:\underline{c}_n\,,\qquad \forall\, n \text{ in } \N.
\end{equation}
Now we consider the following Cauchy problem 
\begin{equation}
\label{oned06}
\begin{cases}
\dis -y_n''(t) =\frac{\chi_{(0,1)}}{c_n(n-1)y_n^n(t)} & \mbox{ for } t \geq 0\,, \\
y_n(0)=1, \\
y_n'(0)=0.
\end{cases}
\end{equation}
For every $t$ in $(0,1)$ we have that \eqref{oned01} and \eqref{oned06} are the same problem, so that there exists $y_n(t)\equiv w_n(t)$ classical solution of \eqref{oned06} in $(0,1)$. Since $y_n''(t)=0$ for every $t\geq 1$, we deduce that $y_n(t)=y_n(1)+y_n'(1)(t-1)=w_n(1)+w_n'(1)(t-1)$ in $[1,2)$. It follows from \eqref{oned02} and by the definition of $\alpha_n$ that
\begin{equation}
\label{oned07}
w_n'(1)=-\sqrt{\frac{2}{c_n(n-1)^2}}\,(w_n^{1-n}(1)-1)^{\frac{1}{2}}\,.
\end{equation}
Since we want that $y_n(2)=0$ for every $n$ in $\N$, we look for $c_n$ such that $w_n'(1)=-w_n(1)$. With a little algebra it follows from \eqref{oned07} and \eqref{oned03} that is possible if and only if, for every fixed $n$, we have
\begin{equation}
\label{oned08}
w_n^{n+1}(1)=\frac{2}{c_n(n-1)^2}\,(1-w_n^{n-1}(1))\,=\,\frac{2}{c_n(n-1)^2}\,S_n^{-1}\left(\sqrt{\frac{2}{c_n}}\right)\,.
\end{equation}
By Lemma \ref{sucn} below there exists a sequence $\{c_n\}$ such that \eqref{oned08} holds for every $n$, hence we have that $y_n$ belonging to $C^1((0,2))$ is such that
\begin{equation}
\label{oned011}
y_n(t)\equiv w_n(t) \text{ in } [0,1]\,,\quad y_n(t)=w_n(1)(2-t) \text{ in }(1,2]\,, \quad y'_n(0)=y_n(2)=0\,.
\end{equation}
We want that $w_n(t)\leq y_n(t)$ in $[0,T_n]$. This is true if and only if $T_n\leq 2$. If, by contradiction, $T_n>2$ we have $w_n(t)\equiv y_n(t)$ in $[0,1]$ and $-y''_n(t)<-w''_n(t)$ in $(1,2]$, so that, by $w'_n(1)=y'_n(1)$, we deduce $w_n(t)< y_n(t)$ in $(1,2]$. It follows from $y_n(2)=0$ that $0<w_n(2)<0$, that is a contradiction. Then we obtain $T_n\leq 2$, $w_n(t)\leq y_n(t)$ in $[0,T_n]$ and, by \eqref{oned04}, that
\begin{equation}
\label{oned09}
c_n\,\leq\,\frac{8\,\Gamma^2\left(\frac{n}{n-1}\right)}{\pi\,\Gamma^2\left(\frac{1}{2}+\frac{1}{n-1}\right)}=:\overline{c}_n\,,\qquad \forall\, n \text{ in } \N\,.
\end{equation}
Thus $\{c_n\}$ is bounded and, up to subsequences, there exists a positive real number $c_\infty$ such that
$$
\frac{2}{\pi^2}\,=\,\lim_{n\to+\infty}\underline{c}_n\,\leq\,c_\infty:=\lim_{n\to+\infty}c_n\,\leq\,\lim_{n\to+\infty}\overline{c}_n\,=\,\frac{8}{\pi^2}\,,
$$
and, respectively,
$$
1\,\leq\,T_\infty:=\lim_{n\to+\infty} T_n\,=\,\pi\,\sqrt{\frac{c_\infty}{2}}\,\leq\,2\,.
$$
As shown in the previous section, it follows from \eqref{oned03} that
\begin{equation}
\label{oned010}
\lim_{n\to+\infty} w^{n+1}_n(t)=\cos^2{\left(\frac{\pi}{2\,T_\infty}t\right)} \quad\text{and}\quad\lim_{n\to+\infty} w_n(t)=1, \qquad\text{for } t\in(0,T_\infty)\,.
\end{equation}
Now we suppose that $T_\infty>1$. Fix $\dis \beta=\frac{T_\infty-1}{2}>0$, so that $1+\beta<T_\infty$. We know that for $n$ large enough
$$
w_n(1+\beta)\,\leq\, y_n(1+\beta)\,=\,w_n(1)(1-\beta)\,.
$$
By passing to the limit as $n$ tends to infinity and using \eqref{oned010} we obtain $1\,\leq \,1-\beta$, that is $\beta\,\leq\,0$. This is a contradiction, then $T_\infty\,=\,1$ and, therefore, $\dis c_\infty\,=\,\frac{2}{\pi^2}$. \\
Recalling that $y_n(t)\equiv w_n(t)$ in $(0,1)$ and using, once again, \eqref{oned010} we have
\begin{equation}
\label{oned012}
\lim_{n\to+\infty} y^{n+1}_n(t)=\cos^2{\left(\frac{\pi}{2}t\right)} \quad\text{and}\quad\lim_{n\to+\infty} y_n(t)=1\,, \qquad\text{for } t\in(0,1)\,.
\end{equation}
It follows from \eqref{oned08} and using that $y_n(1)=w_n(1)$ for every $n$ that 
\begin{equation}
\label{oned013}
\lim_{n\to+\infty}y_n^{n+1}(1)=0\quad\text{and}\quad\lim_{n\to+\infty} y_n(1)=1\,,
\end{equation}
hence, by \eqref{oned011}, we obtain that $y_n^{n+1}(t)=w_n(1)^{n+1}(2-t)^{n+1}$ and that
\begin{equation}
\label{oned014}
\lim_{n\to+\infty} y^{n+1}_n(t)\,=\,0 \quad\text{and}\quad\lim_{n\to+\infty} y_n(t)=(2-t)\,, \qquad\text{for } t\in(1,2]\,.
\end{equation}
Therefore, by the initial condition $y'_n(0)=0$, we can extend $y_n$ to an even function defined in $(-2,2)$ as follows
$$
\tilde{y}_n(t)=
\begin{cases}
(c_n(n-1))^{\frac{1}{n+1}}y_n(t) & \mbox{for $t\in [0,2]$}, \\
(c_n(n-1))^{\frac{1}{n+1}}y_n(-t) & \mbox{for $t\in [-2,0)$},
\end{cases}
$$
so that $\tilde{y}$ belonging to $C^1_0((-2,2))$ is a weak solution of \eqref{oned00}. By Remark \ref{bocas} there is a unique weak solution of \eqref{oned00}, hence $\tilde{y}_n(t)\equiv u_n(t)$ for every $t$ in $(-2,2)$ and $n$ in $\N$. \\
Moreover, by Proposition \ref{link}, setting $\dis v_n(t)=\frac{u_n^{n+1}(t)}{n+1}$, we have that $v_n$ in $C^1_0((-2,2))$ is a weak solution of \eqref{pbv5} and, by Remark \ref{vnbuo}, that there exists a function $v$ such that $v_n$ weakly converges to $v$ in $W^{1,2}_0((-2,2))$ and almost everywhere in $(-2,2)$. As a consequence of \eqref{oned012}, \eqref{oned013} and \eqref{oned014} we deduce that
$$
v(t)=
\begin{cases}
\dis \frac{2}{\pi^2}\cos^2{\left(\frac{\pi}{2}\,t\right)} & \mbox{for $t\in (-1,1)$}, \\
0 & \mbox{for $t\in [-2,-1]\cup[1,2]$},
\end{cases}
$$
so that $v$ belongs to $C^1_0(-2,2)\cap C^{\infty}_0(-1,1)$. Furthermore, with a little algebra, it follows that $v$ is a classical solution of \eqref{pbv6}.
\end{proof}

\begin{remark}
\label{recov}
From the proof of Theorem \ref{main4} we deduce that $u_n$ pointwise converges to $u$ defined as follows
$$
u(t)=
\begin{cases}
(2-t) & \mbox{for $t\in [1,2]$}, \\
1 & \mbox{for $t\in (-1,1)$}, \\
(2+t) & \mbox{for $t\in [-2,-1]$}.
\end{cases}
$$
Moreover, by Theorem \ref{main1}, $u_n$ weakly converges to $u$ in $W^{1,2}_0((-2,2))$. Hence we have that
$$
u'(t)=
\begin{cases}
-1 & \mbox{for $t\in (1,2)$}, \\
0 & \mbox{for $t\in (-1,1)$}, \\
1 & \mbox{for $t\in (-2,-1)$},
\end{cases}
$$
and $u$ is a distributional solution of 
$$
\begin{cases}
- u''=-\delta_{-1}\,+\,\delta_{1} & \mbox{in $(-2,2)$,} \\
\hfill u(\pm 2)= 0\,.
\end{cases}
$$
So that we have completely recovered the results of Theorem \ref{main1}.
\end{remark}

To be complete we show the technical lemma that we needed to prove the theorem. 

\begin{lemma}
\label{sucn}
Let $c$ belong to $(c_0,+\infty)$, with $\dis c_0=\frac{2\,\Gamma^2\left(\frac{n}{n-1}\right)}{\pi\,\Gamma^2\left(\frac{1}{2}+\frac{1}{n-1}\right)}$. Let $w_c(t)$ be the classical solution of 
\begin{equation}
\label{sucn1}
\begin{cases}
\dis -w_c''(t) = \frac{1}{c(n-1)\,w_c^{n}(t)} & \mbox{ for } t\geq 0, \\
w_c(0) =1, \\
w_c'(0)=0. 
\end{cases}
\end{equation}
Let $T_c$ be the first zero of $w_c$. Then there exists a unique $\tilde{c}$ in $(c_0,+\infty)$ such that $T_{\tilde{c}}>1$ and
\begin{equation}
\label{sucn2}
w_{\tilde{c}}^{n+1}(1)=\,\frac{2}{\tilde{c}(n-1)^2}\,S_{\tilde{c}}^{-1}\left(\sqrt{\frac{2}{\tilde{c}}}\right)\,,
\end{equation} 
where $S_c:[0,1]\to [0,S_c(1)]$ is defined as
$$
S_c(1-w_c^{n-1}(t)):=\int_0^{1-w_c^{n-1}(t)}\frac{1}{h^{\frac{1}{2}}\,(1-h)^{\frac{n-3}{2(n-1)}}}\,\,dh\,,
$$
for $t$ in $[0,T_c]$.
\end{lemma}

\begin{proof}
It follows from the proof of Theorem \ref{main4} that if $c>c_0$ then there exists $w_c(t)$ classical solution of \eqref{sucn1} in $[0,T_c]$, with $T_c>1$. \\
Now we define $F:(c_0,+\infty)\to\R$ as 
$$
F(c)\,=\,w_c^{n+1}(1)-\,\frac{2}{c(n-1)^2}\,S_c^{-1}\left(\sqrt{\frac{2}{c}}\right)\,.
$$
It is obvious that $w_c(t)$ is continuous on $(c_0, +\infty)$ for every $t$ in $[0,T_c)$, so that $F$ is continuous. Fix $c_0<c_1<c_2$. Recalling that 
$$
T_c\,=\,\sqrt\frac{\pi\,c}{2}\frac{\Gamma\left(\frac{1}{2}+\frac{1}{n-1}\right)}{\Gamma\left(\frac{n}{n-1}\right)},
$$
we deduce $\dis T_{c_1}<T_{c_2}$. Moreover we state that $\dis w_{c_1}(t)<w_{c_2}(t)$ for every $t$ in $\dis (0,T_{c_1}]$. Indeed, since $-w''_{c_1}(t)>-w''_{c_2}(t)$ near $t=0$ and using the initial conditions, we obtain that $w_{c_1}(t)<w_{c_2}(t)$ near $t=0$. If, by contradiction, there exists $s$ in $(0,T_{c_1})$ such that $w_{c_1}(s)=w_{c_2}(s)$ we have that $w'_{c_1}(s)\geq w'_{c_2}(s)$. We know, by \eqref{oned02}, that 
$$
w_{c_1}'(s)=-\sqrt{\frac{2}{(n-1)^2c_1}}\,(w_{c_1}^{1-n}(s)-1)^{\frac{1}{2}}\,<\,-\sqrt{\frac{2}{(n-1)^2c_2}}\,(w_{c_1}^{1-n}(s)-1)^{\frac{1}{2}}=w_{c_2}'(s)\,,
$$
that is a contradiction. Hence we have that $w_c(t)$ is monotone increasing in $c$. This implies that $F$ also is monotone increasing in $c$. By letting $c$ tend to the boundary of $(c_0,+\infty)$ and recalling that
$$
\lim_{c\to c_0} w_c^{n+1}(1)\,=0\,\quad\text{and}\quad\lim_{c\to+\infty} w_c^{n+1}(1)\,=\,1\,,
$$
we deduce 
$$
\lim_{c\to c_0} F(c)\,=\,-\frac{2}{c_0(n-1)^2}\,S_{c_0}^{-1}\left(\sqrt{\frac{2}{c_0}}\right)\,<\,0 \quad\text{and}\quad\lim_{c\to+\infty} F(c)\,=\,1.
$$
Applying Bolzano's theorem we obtain that there exists $\tilde{c}$ such that $F(\tilde{c})=0$, that is \eqref{sucn2}. Since $F$ is monotone increasing, $\tilde{c}$ is unique.
\end{proof}

\section{Open problems}
\label{s7}

We are now studying the nonexistence of positive solutions of \eqref{pbv4} in the $N$-dimensional case with $f$ only {\sl nonnegative}. More precisely we assume that $f$ is a nonnegative $\elle\infty$ function and that there exists $\omega \subset\subset \Omega$ such that $f = 0$ in $\Omega \setminus \omega$, and such that for every $\omega' \subset \subset \omega$ there exists $c_{\omega'} > 0$ such that $f \geq c_{\omega'}$ in $\omega'$. \\
We observe that from Remark \ref{recov} it follows that $u$, given by Theorem \ref{main4}, is a classical solution of
$$
\begin{cases}
-u''= 0 & \mbox {in $(-2,-1)\cup(1,2)$}\,, \\
u(\pm 1)=1\,, \\
u(\pm 2)=0\,.
\end{cases}
$$
Our conjecture is that it is true also for $N>1$. More precisely we think that the following result holds.
\begin{conj}
Let $u$ be the function given by Theorem \ref{main1}, with $M(x)\equiv I$. Then $u$ is a classical solution of 
$$
\begin{cases}
-\Delta u = 0 & \mbox{in $\Omega \setminus \overline{\omega}$,} \\
\hfill u = 1 \hfill & \mbox{on $\partial\omega$,} \\
\hfill u = 0 \hfill & \mbox{on $\partial\Omega$.}
\end{cases}
$$
\end{conj}
With a similar idea we think that Theorem \ref{main4} holds for $N>1$.
\begin{conj}
\label{conge1}
Let $u_n$ be the solution of \eqref{pbn} given by Theorem \ref{daBO}, with $M(x)\equiv I$. Let $\dis\left\{ v_n = \frac{u_n^{n+1}}{n+1}\right\}$ be the sequence of solutions of \eqref{parn}. Then $\{v_n\}$ is bounded in $\w\cap \linf$, so that it converges, up to subsequences, to a bounded nonnegative function $v$. Moreover $v$ is a weak solution of 
$$
\begin{cases}
\dis -\Delta v +\frac{|\nabla v|^{2}}{v} = f & \mbox{in $\omega$,} \\
\hfill v = 0 \hfill & \mbox{on $\partial\omega$,}
\end{cases}
$$
and $v\equiv 0$ in $\Omega\setminus\omega$.
\end{conj}

\end{document}